\documentclass{article}

 \usepackage{amsmath,amssymb,amsthm,bbm}
 \usepackage{enumitem}
 \usepackage{lineno, hyperref}
 \usepackage{tikz,float}
 \usetikzlibrary{backgrounds}

\usepackage{sectsty}

\allsectionsfont{\raggedright}

\newtheorem{thm}{Theorem}
\newtheorem{lem}[thm]{Lemma}

\newtheorem{rmk}{Remark}
\newtheorem{definition}{Definition}

\newcommand{\one}{\mathbbm{1}}
\newcommand{\PP}{{\mathbbm{P}}}
\newcommand{\RR}{\mathbbm{R}}

\newcommand{\EE}{\mathbbm{E}}
\newcommand{\ud}{{\mathrm{d}}}

\newcommand{\mm}{\mathbbm{m}}
\newcommand{\DD}{\mathbbm{D}_{1,2}}

\newcommand{\bor}{\mathcal{B}}
\newcommand{\ftn}{{\cal F}}

\newcommand{\pink}{}

\begin{document}

\title{Malliavin smoothness on the L\'evy space with H\"older continuous or $BV$ functionals}

\author{
	Eija Laukkarinen\\ {\normalsize E-mail:~\textsf{eija.laukkarinen@jyu.fi}}	\\
	{\normalsize Department of Mathematics and Statistics, }\\
	{\normalsize University of Jyv\"{a}skyl\"{a}, Finland }
}

\date{October 22, 2019}

\maketitle

\begin{abstract}
	We consider Malliavin smoothness of random variables $f(X_1)$, where $X$ is a pure jump L\'evy process and  the function $f$
	is either bounded and H\"older continuous or of bounded variation. We show that Malliavin differentiability and fractional 
	differentiability of $f(X_1)$ depend both on the regularity of $f$ and the Blumenthal-Getoor index of the 
	L\'evy measure.
	\smallskip
 
	{\parindent0em {Keywords:} L\'{e}vy process, Malliavin calculus, interpolation}\\
	{\parindent0em {AMS2010 Subject Classification:}
	60G51, 
	60H07 
	}
\end{abstract}


\section{Introduction}


Consider a L\'evy process $Y$ and the according Malliavin Sobolev space $\DD$ based 
on the It\^{o} chaos decomposition on the L\'evy space of square integrable random variables. 
We recall the space $\DD$ in Section \ref{section:chaos}. We are interested in the ways that
Malliavin differentiability of $f(Y_1)$ depends on the properties of $f$  and the properties of $Y$.

The process $Y$ consists of three components
\[
	Y_t = \gamma t + \sigma B_t + X_t,
\]
where $\gamma,\sigma\in\RR$, $B$ is a  standard Brownian motion and $X$ is a pure jump process. 
For the Brownian motion we have that $f(B_1) \in \DD$  if and only if 
$f\in W^{1,2}(\RR;\PP_{ B_1})$ (see, for instance, 
 Nualart \cite[Exercise 1.2.8]{nualart1995}).
We also examine fractional 
differentiability which is determined by the real interpolation spaces $(L_2(\PP),\DD)_{\theta,q}$ 
between $L_2(\PP)$ and $\DD$ (see Section \ref{section:interpolation}). The fractional 
smoothness of $f(B_1)$ means that $f$ is in a weighted Besov space (see S. Geiss and Hujo \cite{geiss-hujo}, for example). 
In this paper we focus on the pure jump L\'evy process  with $\gamma =0$ and $\sigma=0$. We search for 
properties of  the function $f$ and the L\'evy measure $\nu$  of $X$, which are related to  the smoothness of $f(X_1)$. 
It turns out that Malliavin 
smoothness is in connection to the Blumenthal-Getoor index 
\[
	\beta = \inf \{\xi \geq 0: m_\xi < \infty \}, \quad \text{where}\quad
	m_\xi := \int_\RR \left( |x|^\xi \wedge 1 \right) \nu (\ud x).
\]
We show that the smaller the index $\beta$ is, the higher smoothness of $f(X_1)$ we have for a given $f$ which is 
 H\"older continuous or of bounded variation.

So far little is known about the question for which $f$ and for which $\nu$ one 
has  $f(X_1) \in \DD$ or $f(X_1) \in (L_2(\PP),\DD)_{\theta,q}$. The note \cite{laukkarinenCPP} enlightens 
the case where $\nu(\RR)<\infty$: Then 
\[
	f(X_1)\in\DD \quad \text{if and only if}  \quad \EE \big[ f^2(X_1) \big( N((0,1]\times\RR) + 1 \big) \big]< \infty
\] 
and  
\[
	f(X_1)\in(L_2(\PP),\DD)_{\theta,2} \quad \! \text{if and only if} \! \quad \EE \left[ f^2(X_1) \! \left( N((0,1]\times\RR)^\theta \! + 1 \right) \right]< \infty,
\] 
where $N$ is the Poisson 
random measure associated with $X$ (see Section \ref{section:preliminaries}).

A L\'evy measure $\nu$ always satisfies the property  $m_2 < \infty$, and from Sol\'e, Utzet and Vives 
\cite{sole-utzet-vives} we know that 
\[
	\|f(X_1)\|^2_{\DD} = \|f(X_1)\|^2_{L_2(\PP)} + \int_\RR \EE\left[ \left( f(X_1 + x) - f(X_1) \right)^2 
	\right] \nu(\ud x).  
\]
Since $m_2<\infty$, it follows that $f(X_1) \in \DD$  for any $f$ that is Lipschitz continuous and bounded. On the other hand, if the L\'evy measure $\nu$ is finite, 
then it is sufficient 
that $f$ is bounded to have $f(X_1) \in \DD$. In Section \ref{section:Holder} we shall examine intermediate 
cases, namely that $f$ is bounded and H\"older continuous, that is, in $C^\alpha_b$. In Theorem 
\ref{th:holder_functions} we prove that 
\[ 
	f(X_1)\in\DD \text{ for all } f\in C^\alpha_b \quad \text{if and only if}\quad m_{2\alpha} < \infty,
\]
where the necessity of the condition $m_{2\alpha}<\infty$ holds under assumption \ref{A.pos} given in Section \ref{section:densities}.
 For fractional smoothness we obtain in Theorem \ref{th:Holder_fractional} for $0<\alpha\leq \theta < 1$, that
\[
	 f(X_1)\in(L_2(\PP),\DD)_{\theta,\infty} \text{ for all }f\in C^\alpha_b\quad \text{if} \quad m_{2\alpha/\theta}<\infty,
\]
and under assumption \ref{A.pos2}, that
\[
	f(X_1)\in(L_2(\PP),\DD)_{\theta,\infty} \text{ for all }f\in C^\alpha_b \quad \text{only if} \quad m_{2\alpha/\theta+\varepsilon}<\infty
\]
for all $\varepsilon>0$.  
In Section \ref{section:stable} we see that if the process $X$ is strictly stable and symmetric and $2\alpha/\theta$ is equal to the Blumenthal-Getoor index $\beta$, 
then  $f(X_1)\in(L_2(\PP),\DD)_{\theta,\infty}$ 
 for all $f\in C^\alpha_b$ eventhough $m_{2\alpha/\theta} = m_{\beta} = \infty$.

We also consider normalized functions of bounded 
variation ($NBV$, see Section \ref{section:BV}). In Theorem 
\ref{th:BV_gives_DD} we prove that under assumptions \ref{A.bdd} and \ref{A.pos} it holds that
\[
	f(X_1)\in\DD \text{ for all } f\in NBV \quad \text{if and only if}\quad m_{1} < \infty.
\]
In \cite[Section 4.2]{geiss-geiss-laukkarinen} it was shown that 
$\one_{(K,\infty)}(Y_1)\in (L_2(\PP),\DD)_{1/2,\infty}$, when $Y_1$ has a bounded density. We obtain a sharper smoothness index
for the pure jump process:
Theorem \ref{th:BV_fractional} states that under assumption \ref{A.bdd} it holds that
\[
	 f(X_1)\in(L_2(\PP),\DD)_{\theta,\infty} \text{ for all }f\in NBV \quad \text{if} \quad m_{1/\theta}<\infty,
\]
and under assumption \ref{A.pos2} it holds that
\[
	f(X_1)\in(L_2(\PP),\DD)_{\theta,\infty} \text{ for all }f\in NBV \quad \text{only if} \quad m_{1/\theta+\varepsilon}<\infty
\]
for all $\varepsilon>0$.  
In Section \ref{section:stable} we see that if the process $X$ is strictly stable and symmetric and $1/\theta = \beta$, then  
$f(X_1)\in(L_2(\PP),\DD)_{\theta,\infty}$  for all $f\in NBV$ eventhough $m_{1/\theta} = m_{\beta} = \infty$. 

The method in Section \ref{section:fractional_smoothness} is based on a characterization of fractional smoothness 
which was introduced for the Brownian motion by S. Geiss and Hujo \cite{geiss-hujo}, and which we translate for 
jump processes in Lemma \ref{lemma:characterization_for_fractional_smoothness}.

\subsection{Motivation}

Malliavin smoothness and fractional smoothness play a role for example in discrete approximation of stochastic 
integrals and in the investigation of properties of backward stochastic differential equations (BSDEs):
Consider the orthogonal Galtchouk-Kunita-Watanabe decomposition of $f( Y_1)$, that is,
\[
	f( Y_1) = c + \int_0^1 \varphi_t\, \ud  Y_t + \mathcal{E}.
\]
Then the convergence rate of the equidistant
Riemann-approximation of the integral  depends on the smoothness parameter of $f(Y_1)$. On the other hand, if 
$f( Y_1)$ admits fractional smoothness, then it is possible to adjust the discretization points to obtain the 
best possible convergence rate. (See Geiss et al. \cite{geiss-geiss-laukkarinen}.) The $L_p$-variation of the 
solution 
of certain BSDEs
depends on the Malliavin fractional smoothness of the 
terminal condition $f(Y_1)$. This was shown with more general terminal conditions for the Brownian motion by 
C. Geiss, S. Geiss and Gobet \cite{geiss-geiss-gobet} and S. Geiss and Ylinen \cite{geiss-ylinen} and for $p=2$ for general $L_2$-L\'{e}vy processes by 
C. Geiss and Steinicke \cite{geiss-steinicke2016}. 


\section{Preliminaries} 
\label{section:preliminaries}


Consider a pure jump L\'evy process $X = (X_t)_{t\geq0}$ with c\`{a}dl\`{a}g paths on a complete probability 
space $({\Omega},{\cal F},\mathbbm{P})$ where ${\cal F}$ is the completion of the sigma-algebra generated by 
$X$. The L\'{e}vy-It\^{o} decomposition of a pure jump L\'evy process is
\[
	X_t = \iint_{(0,t]\times \{ |x|>1\}} x N(\ud s,\ud x) 
	+ \iint_{(0,t]\times\left\{0<|x|\leq1\right\}}x\tilde{N}(\ud s,\ud x),
\]
where $N$ is a Poisson random measure on $\mathcal{B}([0,\infty)\times\mathbbm{R})$ and 
$\tilde{N}(\ud s,\ud x) = N(\ud s,\ud x)-\ud s\nu(\ud x)$ is the 
compensated Poisson random measure. The measure $\nu:\mathcal{B}(\mathbbm{R})\to[0,\infty]$ is the L\'{e}vy 
measure of $X$ satisfying $\nu(\{0\})=0$, $\int_\mathbbm{R} (x^2\wedge1)\nu(\ud x)<\infty$ and 
$\nu(B)=\EE \left[ N((0,1]\times B) \right]$.

\subsection{It\^{o} chaos decomposition and the Malliavin Sobolev space} \label{section:chaos}

Denote $\mathbbm{R}_+ := [0,\infty)$. We consider the measure 
$\mathbbm{m}\colon\mathcal{B}(\mathbbm{R}_+\times\mathbbm{R})\to[0,\infty]$ defined as 
\[ 
	 \mm(A) :=  \int_A x^2 \ud t \nu(\ud x) = \EE \left[ \left(   \int_A x \tilde{N}(\ud t, \ud x) \right)^2 \right].
\]
For $n=1,2,\ldots$ we write $L_2(\mathbbm{m}^{\otimes n}) := L_2 \left((\RR_+\times\RR)^n, 
\bor(\RR_+\times\RR)^{\otimes n}, \mathbbm{m}^{\otimes n}\right)$ and set 
$L_2(\mathbbm{m}^{\otimes 0}):=\mathbbm{R}$. A function $f_n:(\mathbbm{R}_+\times\mathbbm{R})^n\to\mathbbm{R}$ 
is said to be symmetric, if it coincides with its symmetrization $\tilde{f}_n$,
\[
	\tilde{f}_n((s_1,x_1),\ldots,(s_n,x_n))=\frac{1}{n!}\sum_{\pi}f_n\left((s_{\pi(1)},x_{\pi(1)}),\ldots,
	(s_{\pi(n)},x_{\pi(n)})\right),
\]
where the sum is taken over all permutations $\pi:\{1,\ldots,n\}\to\{1,\ldots,n\}$.
			
We consider It\^{o}'s multiple stochastic integral $I_n:L_2(\mathbbm{m}^{\otimes n})\to L_2(\mathbbm{P})$ of 
order $n$ with respect to the measure $x\tilde{N}(\ud t,\ud x)$. According to \cite[Theorem 2]{ito} it holds 
that
\[
	L_2(\mathbbm{P}) =  \RR \oplus  \bigoplus_{n=1}^\infty \{I_n(f_n):f_n\in L_2(\mathbbm{m}^{\otimes n})\}.
\]
The functions $f_n$ in the representation $F=\sum_{n=0}^\infty I_n(f_n)$ in $L_2(\mathbbm{P})$ are unique when 
they are chosen to be symmetric, which is always possible since $I_n(f_n) = I_n(\tilde{f_n})$. Moreover, we have
\[
	\EE\left[ I_n(f_n)I_k(g_k)\right] = \begin{cases}
	0, & \text{ if }n\neq k\\
	 n!(\tilde{f_n},\tilde{g_n})_{L_2(\mm^{\otimes n})} & \text{ if }n = k
	\end{cases}
\]
and
\[
	\|F\|^2_{L_2(\mathbbm{P})} = \sum_{n=0}^\infty n! 
	\left\| \tilde{f}_n \right\|^2_{L_2(\mathbbm{m}^{\otimes n})}.
\]

In this paper we focus on  random variables of the form $f(X_1)$, where $f:\RR\to\RR$ is a Borel 
function. We will take advantage of the following lemma in Sections \ref{section:Holder} and \ref{section:fractional_smoothness}.

\begin{lem}\label{lemma:chaos_martingale}
	Let  $f( X_1) = \sum_{n=0}^\infty I_n(f_n) \in L_2(\PP)$ and let $(\ftn_t)_{t\geq 0}$ be 
	the augmented natural filtration of $X$. Then 
	\begin{itemize}
		\item[(a)] there are functions $g_n\in L_2\left((x^2\nu(\ud x))^{\otimes n}\right)$ such that 
		\[
			\tilde{f_n}((t_1,x_1),\ldots,(t_n,x_n)) = g_n(x_1,\ldots,x_n)
			\one_{[0,1]^{\times n}}(t_1,\ldots,t_n)
		\]
		for $\mm^{\otimes n}$-a.e. $((t_1,x_1),\ldots,(t_n,x_n))\in (\RR_+\times\RR)^{\times n}$ and
		\item[(b)] $\EE \left[ \EE \left[ f(X_1) | \ftn_t  \right]^2\right] = \sum_{n=0}^\infty t^n  n! 
		\|\tilde{f_n}\|^2_{L_2(\mm^{\otimes n})}$.
	\end{itemize}
\end{lem}

\begin{proof}
(a)  Follows from \cite[Remark 6.7]{baumgartner-geiss}.
(b) By analogous argumentation to \cite[Lemma 1.2.4]{nualart1995} we see that
	$\EE \left[ f(  X_1) | \ftn_t  \right] = \sum_{n=0}^\infty I_n(g_n\one_{[0,t]^{\times n}})$. The claim follows 
	from $\|\tilde{f_n}\|_{L_2(\mm^{\otimes n})} = \|g_n\|_{L_2( (x^2\nu(\ud x))^{\otimes n})}$.
\end{proof}

We define the Malliavin Sobolev space using It\^{o}'s chaos decomposition (as \cite{nualart-vives, 
dinunno-meyerbrandis-oksendal-proske, sole-utzet-vives, sole-utzet-vives_chaoses, applebaum2, 
geiss-laukkarinen} and many others). We denote by $\mathbbm{D}_{1,2}$ the space of all $F = \sum_{n=0}^\infty
I_n(f_n) \in L_2(\mathbbm{P})$ such that
\[
	\|F\|^2_{\mathbbm{D}_{1,2}} :=  \|F\|^2_{L_2(\PP)} + \sum_{n=1}^\infty n n! \left\| \tilde{f}_n \right\|^2_{L_2(\mathbbm{m}^{\otimes n})} 
	= \sum_{n=0}^\infty (n+1)! \left\| \tilde{f}_n \right\|^2_{L_2(\mathbbm{m}^{\otimes n})}
	< \infty.
\]
Let us write $L_2(\mathbbm{m}\otimes\mathbbm{P}) := L_2(\mathbbm{R}_+\times\mathbbm{R}\times{\Omega},
\mathcal{B}(\mathbbm{R}_+\times\mathbbm{R})\otimes{\cal F}, \mathbbm{m}\otimes\mathbbm{P})$. The
Malliavin derivative $D:\mathbbm{D}_{1,2}\to L_2(\mathbbm{m}\otimes\mathbbm{P})$ is defined for $F\in\DD$ by
\[
	D_{t,x}F = \sum_{n=1}^\infty n I_{n-1} (\tilde{f}_n(\cdot,(t,x))) \quad \textrm{in }
	L_2(\mathbbm{m}\otimes\mathbbm{P}).
\]

From \cite[Proposition 5.4]{sole-utzet-vives} we have in the canonical probability space that 
\begin{align}\label{eq:quotient_in_L_2}
	& \|f(X_1)\|^2_{\DD} \nonumber \\
	& = \|f(X_1)\|^2_{L_2(\PP)} + \int_{[0,1]\times\RR\setminus\{0\}} \EE \left[ \left( \frac{ f(X_1+x) - f(X_1) }{x}  \right)^2 \right] \mm(\ud t, \ud x) \nonumber \\
	& = \|f(X_1)\|^2_{L_2(\PP)} + \int_\RR \EE\left[ \big(f(X_1+x)-f(X_1)\big)^2 \right] \nu(\ud x),
\end{align}
and  when $f(X_1)\in\DD$, then
\begin{equation}\label{eq:quotient}
	D_{t,x}f(X_1) = \frac{ f(X_1+x) - f(X_1) }{x} \one_{[0,1]\times \RR\setminus\{0\}}(t,x) \quad \mm\otimes\PP\text{-a.e.}
\end{equation}
 The result was converted to the general probability space in \cite[Lemma 3.2]{geiss-steinicke}.
  
For the Brownian motion $B$, the space $\DD$ is defined in an analogous way by a chaos decomposition, but the property 
\eqref{eq:quotient_in_L_2}  can not be formulated (see \cite{nualart1995}).

\subsection{Interpolation and Malliavin fractional smoothness} \label{section:interpolation}

The interpolation space $(A_0,A_1)_{\theta,q}$ is a Banach space, intermediate between two Banach spaces $A_0$ 
and $A_1$ which are a compatible couple, that is, they are continuously embedded into a Hausdorff topological 
vector space.
			
When $(A_0,A_1)$ is a compatible couple, the \emph{K-functional} of $ a\in A_0 + A_1$ is the mapping
$K( a, \cdot\,  ;A_0, A_1) : (0,\infty) \to [0,\infty)$ defined by 
\[ 
	K( a,t; A_0, A_1) := \inf\{ \| a_0\|_{A_0} + t\| a_1\|_{A_1}:\  a= a_0+ a_1,\ a_0\in A_0,\ a_1\in A_1 \}.
\]
Let $\theta\in(0,1)$ and $q\in[1,\infty]$. The \emph{real interpolation space} $(A_0,A_1)_{\theta,q}$ consists of all 
$a\in A_0 + A_1 :=\{a_0+a_1 : a_0\in A_0,\, a_1\in A_1\}$
such that  the norm
\[
	\|a\|_{(A_0,A_1)_{\theta,q}} =  \left\{
	\begin{aligned}
		& \left[ \int_0^\infty \left( t^{-\theta} K( a,t; A_0, A_1) \right)^q \frac{\ud t}{t} \right]^{\frac{1}{q}},
		                                         & q \in[1,\infty) \\
		& \sup_{t>0} t^{-\theta} K(a,t ; A_0,A_1), &   q  = \infty \hspace{1.4em}    
	\end{aligned} \right.
\]
is finite. If $A_1 \subseteq A_0$ with continuous embedding, then 
\begin{equation}\label{eq:lexicographical_order}
	A_1 \subseteq (A_0,A_1)_{\theta,q} \subseteq (A_0,A_1)_{\eta,p} \subseteq (A_0,A_1)_{\eta,q} \subseteq A_0
\end{equation}
for $ 0 < \eta < \theta < 1$ and $1\leq p \leq q \leq \infty$.

From the Reiteration Theorem we know that for 
$\eta,\theta\in(0,1)$ and $q\in[1,\infty]$ one has
\begin{equation}
	\label{eq:reiteration}
	(A_0,(A_0,A_1)_{\eta,\infty})_{\theta,q} = (A_0,A_1)_{\eta\theta,q}
\end{equation}
with
\begin{equation}
	\label{eq:reiteration_norms} 
	\| a\|_{(A_0,A_1)_{\eta\theta,\infty}} 
	\leq \|a\|_{(A_0,(A_0,A_1)_{\eta,\infty})_{\theta,\infty}} 
	\leq 3 \| a\|_{(A_0,A_1)_{\eta\theta,\infty}}
\end{equation}
for all $a\in (A_0,A_1)_{\eta\theta,\infty} = \left(A_0,(A_0,A_1)_{\eta,\infty} \right)_{\theta,\infty}$.
In the literature the Reiteration Theorem is usually given in a more general context and the constants $1$ and $3$ in the 
norm equivalence \eqref{eq:reiteration_norms} are not computed explicitely. Therefore we verify 
\eqref{eq:reiteration_norms} in Lemma \ref{lemma:constantA}. For further properties of 
interpolation spaces, see for instance \cite{bennett-sharpley}, \cite{bergh-lofstrom} or \cite{triebel}.
				
We say that
a random variable admits fractional smoothness of order $(\theta,q)$ if it belongs to the interpolation space
\[
	 \left(L_2(\PP), \DD  \right)_{\theta,q}, 
\]
where $\theta\in(0,1)$ and $q\in[1,\infty]$.

\subsection{Assumptions about a density}\label{section:densities}

Some of the assertions in this paper rest on the following assumptions:
\begin{enumerate}[label=(\textbf{A\arabic*})]
	\item $X_1$  has a bounded density $p_1$.  \label{A.bdd}
	\item $X_1$  has a density $p_1$  and there exist  $a,b,c \in\RR$ with $c>0$  and $b-a>0$ 
					      such that  $p_1(x) \geq c$  for all $x\in[a,b]$.  \label{A.pos}
	\item There exist $t_0\in(0,1)$ and $a,b,c \in\RR$ with $c>0$ and $b-a>0$ such that for all $t\in[t_0,1]$, the random variable $X_t$ has a density $p_t$
	     such that $p_t(x) \geq c$  for all $x\in[a,b]$.\label{A.pos2}
\end{enumerate}
 Note that the conditions \ref{A.bdd}, \ref{A.pos} and \ref{A.pos2} are satisfied, for example, when
the condition 
\[
	\ell:=\liminf_{|u|\to\infty} \frac{\int_{\RR} \sin^2(ux) \nu(\ud x)}{\log|u|} > \frac12
\]
of Hartman and Wintner \cite{hartman-wintner} holds. We formulate the argumentation in a lemma as it will be used later.

\begin{lem}\label{lem:assumptions}
	Assume that $\ell>1/2$. Then \ref{A.bdd}, \ref{A.pos} and \ref{A.pos2} are satisfied.
\end{lem}
\begin{proof}
	By \cite[Section 13, statement II]{hartman-wintner}, $X_t$ has a bounded and 
	continuous density for all $t>\tfrac{1}{2\ell}$. The conditions \ref{A.bdd} and \ref{A.pos} follow immediately. Let us prove \ref{A.pos2}.
	Let $r>0$. Due to stochastic continuity of L\'evy processes, there is $t_0\in\left( \tfrac1{2\ell},1 \right)$ such that
	\[
		\PP(|X_{t-t_0}|\leq r)\geq 1/2\quad \text{ for all } t\in[t_0,1].
	\]
	Since $\ell>1/2$, \cite[Theorem 24.10]{sato} implies that either 
	 the support of $X_{t_0}$ is a half line $[\kappa,\infty)$ (or $(-\infty,\kappa]$) for some $\kappa \in\RR$, or 
	 the support of $X_s$ is $\RR$ for all $s>0$. 
	 The continuous density $p_{t_0}$, if supported on a half line, is strictly 
	 positive on the open half line $(\kappa,\infty)$ (or $(-\infty,\kappa)$) by  \cite[Chapter IV, Theorem 8.6]{steutel-vanharn}.
	 If $X_s$ has a bounded and continuous density supported on the whole real line
	 for $\tfrac1{2\ell} < s < t_0$, then \cite[Chapter IV, Theorem 8.6]{steutel-vanharn} implies that $p_{t_0}$ is strictly positive.
	In any case $p_{t_0}$ is continuous and strictly positive on at least a half line, so that we find $K\in\RR$ and $c >0$ such that $p_{t_0}(x) \geq c$ for all $x\in[K-2r,K+2r]$. 
	For any $x\in[K-r,K+r]$ and $t\in[t_0,1]$ it holds that
	\begin{align*}
		p_t(x)& = \int_\RR p_{t_0}(x-y) \PP_{X_{t-t_0}}(\ud y)
		       \geq \int_{[-r,r]} p_{t_0}(x-y) \PP_{X_{t-t_0}}(\ud y)\\
		      & \geq c\PP(|X_{t-t_0}|\leq r) \geq c/2.
	\end{align*}
\end{proof}


\section{H\"older continuous functions and Malliavin smoothness} 
\label{section:Holder}


For $\alpha\in(0,1]$, the spaces $B(\RR)$,  $C^\alpha$  and $C^\alpha_b$ are spaces of Borel measurable functions $f$ such that 
\[
	\|f\|_\infty = \sup_{x\in\RR} |f(x)|,  \quad  \|f\|_{C^\alpha} =  \sup_{x\neq y} \tfrac{ |f(x) - f(y)| }{ |x-y|^{\alpha} } \quad \text{or} \quad 
	\|f\|_{C_b^{\alpha}} = \|f\|_\infty + \|f\|_{C^\alpha}, 
\]
respectively, is finite. We frequently use the notation $Lip := C^1_b$.  Note that $(B(\RR),\|\cdot\|_\infty)$ and $(C_b^\alpha,\|\cdot\|_{C_b^\alpha})$ are 
Banach spaces and $\|\cdot\|_{C^\alpha}$ is a seminorm.
Recall the notation
$$m_{2\alpha} = \int_\RR \left( |x|^{2\alpha}\wedge 1 \right) \nu(\ud x).$$


\subsection{Smoothness of first order}

\begin{thm}\label{th:holder_functions}
	Let $\alpha \in (0, 1)$  and $A:=[0,1]\times \{x:|x|>1\}$ and assume that $f(X_1)\in L_2(\PP)$. 
	\begin{enumerate}
		\item[(a)] 
			If  $f \in C^\alpha$ and
			$\int_\RR |x|^{2\alpha} \nu(\ud x) < \infty$, then $f(X_1)\in\DD$ and
			\[
				\|f(X_1)\|^2_{\DD} \leq  \| f(X_1) \|^2_{L_2(\PP)} + \|f\|^2_{C^\alpha}\int_\RR |x|^{2\alpha} \nu(\ud x) .
			\]
		\item[(b)] 
			If $f\in C^\alpha$, $m_{2\alpha} < \infty$ and $\EE\left[ f^2(X_1) N(A) \right]<\infty$,
			then $f(X_1)\in\DD$ and
			\begin{align*}
				& \|f(X_1)\|^2_{\DD} \\ 
				& \leq \|f\|_{C^\alpha}^2m_{2\alpha} + \EE\left[ f^2(X_1) N(A) \right] + \|f(X_1)\|^2_{L_2(\PP)}(1+\nu(\{|x|>1\})).
			\end{align*}
		\item[(c)] 
			If $f\in C_b^\alpha$ and $m_{2\alpha} < \infty $, then $f(X_1)\in\DD$ and
			\begin{equation} \label{eq:D12Calphanorms}
				\|f(X_1)\|^2_{\DD} \leq  \left(1 + 4 m_{2\alpha}\right)\|f\|^2_{C^\alpha_b}.       
			\end{equation}
		\item[(d)] 
			Assume that \ref{A.pos} holds and choose $\ell\in\{0,1,2,\ldots\}$ such that there exist $k\in\mathbbm{Z}$ and $c>0$ with 
			$p_1(x)\geq c$ for all $x\in\left[ k2^{-\ell}, (k+1)2^{-\ell} \right]$. Then for the function  
			$\displaystyle g^{\alpha,\ell}(x) = \sum_{n=\ell}^\infty 2^{-\alpha n}d(2^nx,\mathbbm{Z})$
			from Lemma \ref{lem:ciesielski_function} it holds that $g^{\alpha,\ell}\in C^\alpha_b$, and 
			\[
				g^{\alpha,\ell}(X_1)\in\DD \quad \text{only if}\quad m_{2\alpha} < \infty.
			\]
	\end{enumerate}
\end{thm}
\begin{proof} 
	(a) The claim follows from \cite[Proposition 5.4]{sole-utzet-vives} (see \eqref{eq:quotient_in_L_2}) and the $\alpha$-H\"older continuity.
	
	(c) The claim follows from $\|f(X_1)\|_{L_2(\PP)}^2 \leq \|f\|_{C_b^{\alpha}}^2$ and \eqref{eq:quotient_in_L_2}, since 
	\begin{align*}
		& \int_\RR  \EE \left[ \left| f(X_1+x) - f(X_1) \right|^2 \right] \nu(\ud x)  \\
		&\leq \int_{\{|x|\leq 1\}} \|f\|^2_{C^\alpha} |x|^{2\alpha} \nu(\ud x) 
		+ \int_{\{|x|>1\}} 4 \|f\|_\infty^2 \nu(\ud x) \\
		& \leq \|f\|_{C_b^{\alpha}}^2 \cdot 4 \int_\RR \left( |x|^{2\alpha} \wedge 1 \right) \nu(\ud x).
	\end{align*}
	
	(b) 
	Consider the chaos expansion $f(X_1) = \sum_{n=0}^\infty I_n(f_n)$ and recall that
	\[
		\|f(X_1)\|^2_{\DD} = \|f(X_1)\|^2_{L_2(\PP)} + \sum_{n=1}^\infty n n! \left\| \tilde{f_n} \right\|^2_{L_2(\mm^{\otimes n})}.
	\] 
	We show first that
	\begin{align}\label{eq:chaos_split}\sum_{n=1}^\infty n n! \left\| \tilde{f_n} \right\|^2_{L_2(\mm^{\otimes n})}   
		& = \int_{[-1,1]}\EE\left[ |f(X_1+x)-f(X_1)|^2 \right] \nu(\ud x) \nonumber\\
		&\quad + \sum_{n=1}^\infty n n! \left\| \tilde{f_n} \one_{(\RR_+\times\RR)^{\times(n-1)}\times A}\right\|^2_{L_2(\mm^{\otimes n})}.
	\end{align}
	In fact, it holds that
        \begin{align}\label{align:C_alpha_bound}
		& \int_{\RR_+\times\RR\setminus\{0\}} \EE \left[ \left| \frac{f(X_1+x) - f(X_1)}{x}\one_{[0,1]\times\{0<|x|\leq 1\}}(t,x) \right|^2  \right] 
		  \mm(\ud t, \ud x)\nonumber \\*
		& = \int_{[-1,1]}\EE\left[ |f(X_1+x)-f(X_1)|^2 \right] \nu(\ud x)
		 \leq \|f\|^2_{C^\alpha} \int_{[-1,1]} |x|^{2\alpha} \nu(\ud x)<\infty,
		\end{align}  so that there is a chaos representation
	\[
		\frac{f(X_1+x) - f(X_1)}{x}\one_{[0,1]\times\{0<|x|\leq 1\}}(t,x) = \sum_{n=0}^\infty I_n(h_{n+1}(\cdot,(t,x))) \quad\text{in }L_2(\mm\otimes\PP)
	\]
	where $h_{n+1}\in L_2(\mm^{\otimes (n+1)}) $ is symmetric in the first $n$ pairs of variables (see \cite[Lemma 1.3.1]{nualart1995} or \cite[Section 4]{nualart-vives}).
	Let $\varphi_k = -k \vee (f \wedge k)$ so that $\varphi_k \in C_b^\alpha$ and
	$\varphi_k(X_1) \in \DD$ by (c). 
	Consider the chaos expansion $\varphi_k(X_1) = \sum_{n=0}^\infty I_n(f_n^{(k)})$. Then $\tilde{f}_n^{(k)} \to \tilde{f}_n$ in $L_2(\mm^{\otimes n})$, 
	since $\varphi_k(X_1) \to f(X_1)$ in $L_2(\PP)$. 
	It also holds that
	\[
		\int_{[0,1]\times\{0<|x|\leq1\}} \EE  \left[ \left| \frac{\varphi_k(X_1 + x) - \varphi_k(X_1)}{x} - 
		\frac{ f(X_1+x) - f(X_1) }{x}  \right|^2  \right] \mm(\ud t,\ud x) 
	\]
	converges to $0$ as $k\to \infty$ by dominated convergence, since $|\varphi_k(X_1 + x) - \varphi_k(X_1)| \leq |f(X_1+x) - f(X_1)|$.  
	From \eqref{eq:quotient} we have that
	\[
		\frac{\varphi_k(X_1+x)\!-\!\varphi_k(X_1)}{x}\one_{[0,1]\times\RR\setminus\{0\}}(t,x) = D_{t,x}\varphi_k(X_1) = \!\sum_{n=1}^\infty n I_{n-1}(\tilde{f}_n^{(k)}(\cdot,(t,x)),
	\] 
	in $L_2(\mm\otimes\PP)$, which gives
	\begin{align*}
		h_n 
		& = \lim_{k\to\infty}n\tilde{f}_n^{(k)}\one_{(\RR_+\times\RR)^{\times(n-1)}\times([0,1]\times\{0<|x|\leq1\})} \\
		& = n\tilde{f}_n\one_{(\RR_+\times\RR)^{\times(n-1)}\times([0,1]\times\{0<|x|\leq 1\})}
	\end{align*}
        in $L_2(\mm^{\otimes n})$ for $n=1,2,\ldots$ Therefore
	\begin{align*}
		& \frac{f(X_1+x) - f(X_1)}{x}\one_{[0,1]\times\{0<|x|\leq 1\}}(t,x) \\
		& = \sum_{n=1}^\infty nI_{n-1}(\tilde{f}_n(\cdot,(t,x))\one_{[0,1]\times\{0<|x|\leq 1\}}(t,x)) 
	\end{align*}
	in $L_2(\mm\otimes\PP)$. This together with Lemma \ref{lemma:chaos_martingale}(a) proves equation \eqref{eq:chaos_split}.
	For the second term on the right hand side of \eqref{eq:chaos_split} we have by \cite[Proposition 3.4]{laukkarinenCPP} that
	\[
		\sum_{n=1}^\infty n n! \left\| \tilde{f_n} \one_{(\RR_+\times\RR)^{\times(n-1)}\times A}\right\|^2_{L_2(\mm^{\otimes n})}
	        \leq \EE\left[ f^2(X_1) N(A) \right] + \EE[f^2(X_1)]\EE[N(A)].
	\]
	Thus, from  \eqref{eq:chaos_split}, \eqref{align:C_alpha_bound} and the above inequality  we 
	get that 
	\begin{align*}
		 \sum_{n=1}^\infty n n! \left\| \tilde{f_n} \right\|^2_{L_2(\mm^{\otimes n})}   
		 \leq \|f\|^2_{C^\alpha} m_{2\alpha} 
		+ \EE\left[ f^2(X_1) N(A) \right] + \EE[f^2(X_1)]\EE[N(A)].
	\end{align*}
	Noting that $\EE[N(A)]=\nu(\{|x|>1\})$, we obtain the claim.
		
	(d)  We have $g^{\alpha,\ell}\in C_b^\alpha$ by Lemma \ref{lem:ciesielski_function} below. If $g^{\alpha,\ell}(X_1)\in\DD$, then by 
	\eqref{eq:quotient_in_L_2} and Lemma \ref{lem:ciesielski_function} it holds that
	\begin{align*}
		  \infty 
		& > \int_\RR \EE\left[ \big( g^{\alpha,\ell}(X_1+x)-g^{\alpha,\ell}(X_1) \big)^2 \right] \nu(\ud x)\\*
		& \geq \int_{|x|\leq 2^{-\ell-3}} \left[ c\int_{k2^{-\ell}}^{(k+1)2^{-\ell}}  \left(g(y+x)-g(y)\right)^2  \ud y \right]\nu(\ud x) \\
		& \geq c2^{-\ell}2^{8\alpha-10} \int_{|x|\leq 2^{-\ell-3}} |x|^{2\alpha}\nu(\ud x).
		     \end{align*}
	Hence it must be $m_{2\alpha}<\infty$.
\end{proof}

The idea for the construction of the function $g^{\alpha,\ell}$ below is based on the decomposition of Ciesielski \cite{ciesielski}. 

\begin{lem}\label{lem:ciesielski_function}
	Let $\ell\in\{0,1,2\ldots\}$ and $\displaystyle g^{\alpha,\ell}(x) = \sum_{n=\ell}^\infty 2^{-\alpha n} g_n(x)$, where 
	\[
		g_n(x)= d(2^nx,\mathbbm{Z}) = \inf\{|2^nx - z|:z\in\mathbbm{Z}\}.
	\] 
	Then $g^{\alpha,\ell}\in C_b^\alpha$, and for all $k\in\mathbbm{Z}$ and $|x|\leq 2^{-\ell-3}$ it holds that
	\[
		\int_{k2^{-\ell}}^{(k+1)2^{-\ell}} \left[ g^{\alpha,\ell}(y+x)-g^{\alpha,\ell}(y) \right]^2 \ud y \geq 2^{-\ell}2^{8\alpha-10}|x|^{2\alpha}.
	\]
\end{lem}
\begin{proof} 
	Since $|g_n(x)|\leq 1/2$ for all $x\in\RR$, it is clear that $\|g^{\alpha,\ell}\|_\infty < \infty$. 
	Since we also have that $|g_n(x)-g_n(y)|\leq 2^n|x-y|$ for all $x,y\in\RR$, we get for any $m\geq \ell$ and $2^{-m-1}\leq|x-y|\leq 2^{-m}$, that
	\begin{align*}
		|g^{\alpha,\ell}(x)-g^{\alpha,\ell}(y)|
		& \leq \sum_{n=\ell}^\infty 2^{-\alpha n}|g_n(x)-g_n(y)| \\
		& \leq \sum_{n=0}^m 2^{-\alpha n}2^n 2^{-m} + \sum_{n=m+1}^\infty 2^{-\alpha n}\\
		& \leq \frac{2(2^{- m -1})^\alpha}{2^{1-\alpha}-1} +  \frac{(2^{-m-1})^\alpha}{1-2^{-\alpha}}\\
		& \leq \left( \frac{1}{(2^{1-\alpha}-1)(1-2^{-\alpha})}\right)|x-y|^\alpha.
	\end{align*}
	Thus $g^{\alpha,\ell}\in C_b^\alpha$.
	
	The function $g_m$ is periodic with period length $2^{-n}$ for all $m\geq n$, so that via dominated convergence we get that
	 \begin{align*}
		&\int_{k2^{-\ell}}^{(k+1)2^{-\ell}} \left[ g^{\alpha,\ell}(y+x)-g^{\alpha,\ell}(y) \right]^2 \ud y\\
		& = \sum_{n=\ell}^\infty 2^{n-\ell-2\alpha n}\int_0^{2^{-n}} 
		        \left[ g_n(y+x)- g_n(y)\right]^2 \ud y  \\
		& \qquad + 2\sum_{m>n\geq \ell}2^{n-\ell-\alpha(n+m)} \int_0^{2^{-n}}  \left[ g_n(y+x)- g_n(y)  \right] \left[ g_m(y+x)- g_m(y)\right]  \ud y.                                                                                                                                        
	\end{align*}
	Let $m>n\geq \ell$.  Since $g_m$ is periodic with period length $2^{-n-1}$ 
	and
	\[
		g_n(y+x)- g_n(y) = - \left( g_n(y+2^{-n-1}+x)- g_n(y+2^{-n-1}) \right)
	\] 
	for all $x,y \in\RR$, we have that
	\begin{align*}
		& \int_0^{2^{-n}}  \left[g_n(y+x)- g_n(y)  \right] \left[g_m(y+x)- g_m(y)\right] \ud y                    \\*
		& =  \int_0^{2^{-n-1}} \left[ g_n(y+x)- g_n(y)  \right] \left[ g_m(y+x)- g_m(y) \right]  \ud y              \\
		& \quad  + \int_{2^{-n-1}}^{2^{-n}} \left[ g_n(y+x)- g_n(y)  \right] \left[ g_m(y+x)- g_m(y) \right]  \ud y \\
		& = 0.                                                                                                                                
	\end{align*}
	Let $0< |x| \leq 2^{-\ell-3}$ and $m\geq\ell$ such that  $2^{-m-4}< |x| \leq 2^{-m-3}$. Since $|g_{m}(y+x)-g_m(y)| =2^m|x|$ when both
	$y+x\in \left(0,2^{-m-1}\right)$ and $y\in \left(0,2^{-m-1}\right)$, we obtain that
	\[
		    \int_0^{2^{-m}}  \left[ g_{m}(y+x)-g_{m}(y) \right]^2 \ud y
		 \geq \int_{2^{-m-3}}^{3\cdot2^{-m-3}}   \left[ 2^{m}|x| \right]^2 \ud y 
		 = 2^{m-2}x^2.
	\]
	Since $2^{m-2}x^2 \geq  2^{m-2}(2^{-m-4})^{2-2\alpha} |x|^{2\alpha} = 2^{-m+2\alpha m +8\alpha-10}|x|^{2\alpha}$, we get	
	\begin{align*}
		  \sum_{n=\ell}^ \infty 2^{n-\ell-2\alpha n}\int_0^{2^{-n}}  \left[ g_{n}(y+x)- g_{n}(y) \right]^2 \ud y     
		  &  \geq  2^{m-\ell-2\alpha m}2^{-m+2\alpha m +8\alpha-10}\\
		  & \geq 2^{-\ell}2^{8\alpha-10} |x|^{2\alpha}.
	\end{align*} 
\end{proof}

\begin{rmk}
        The function $g^{\alpha,\ell}$ in Theorem \ref{th:holder_functions}(d) and Lemma \ref{lem:ciesielski_function} is irregular on the whole real line.
        If a $C_b^\alpha$-function is "more smooth",
        then Theorem \ref{th:holder_functions}(d) does not necessarily give the best condition:  Take for example $f(x)=|x|^\alpha\wedge1$, which is $C_b^\alpha$ 
        but not $C_b^{\alpha'}$
        for any $\alpha'>\alpha$, and assume that \ref{A.bdd} holds. Then
	for $0<|x|\leq 1$ we have that
	\begin{align*}
		& \EE\left[ \left( |X_1+x|^\alpha\wedge1 - |X_1|^\alpha\wedge1  \right)^2 \right]\\
		& \leq \|p_1\|_\infty \int_{-2}^2 \left( |y+x|^\alpha - |y|^\alpha  \right)^2 \ud y\\
		& =  \|p_1\|_\infty|x|^{2\alpha+1} \int_{-\frac{2}{|x|}}^{\frac{2}{|x|}} \left( \left|z+\tfrac{x}{|x|}\right|^\alpha - \left| z\right|^\alpha  \right)^2 
		     \ud z\\
		&  \leq \|p_1\|_\infty |x|^{2\alpha+1} \left[ \int_{|z|<2} 1 \ud z + \alpha^2 \int_{2\leq |z| \leq \frac{2}{|x|}} (|z|-1)^{2\alpha-2} \ud z  \right]\\
		&  \leq \begin{cases}
		        \|p_1\|_\infty |x|^{2\alpha+1} \left[ 4 + \frac{2\alpha^2}{1-2\alpha}  \right], & \text{ for }\alpha < \frac12 \\
		        \|p_1\|_\infty |x|^2 \left[ 4 + 2\log \frac2{|x|}  \right], & \text{ for }\alpha = \frac12 \\
		         \|p_1\|_\infty |x|^{2\alpha+1} \left[ 4 + \frac{2^{2\alpha}\alpha^2}{2\alpha-1} |x|^{1-2\alpha}  \right], & \text{ for }\alpha > \frac12. 
		       \end{cases}
	\end{align*}
        Since $\EE\left[ \left( |X_1+x|^\alpha\wedge1 - |X_1|^\alpha\wedge1  \right)^2 \right] \leq 1$, we get from \eqref{eq:quotient_in_L_2} that 
        $|X_1|^\alpha\wedge1\in\DD$, if one of the following three conditions holds: 
        1. $0<\alpha<1/2$ and $m_{2\alpha+1}<\infty$, 2. $\alpha=1/2$ and $\int_{\{0<|x|\leq1\}} x^2\log (1/|x|) \nu(\ud x)<\infty$ or
        3. $\alpha>1/2$. Note that for the Brownian motion $B$ we have $|B_1|^\alpha\wedge1\in\DD$ if and only if $\alpha>1/2$. This can be easily seen using
        \cite[Example 1.2.8]{nualart1995}.
\end{rmk}


\subsection{Fractional smoothness}

To find fractional smoothness for $f(X_1)$ with $f\in C^\alpha_b$ in Corollary \ref{th:Holder_fractional} below, 
we take advantage of the fact that  $C^\alpha_b = (B(\RR), Lip)_{\alpha,\infty}$ with 
\begin{equation} 
	\label{eq:Holder_interpolation}
	\|\cdot\|_{C^\alpha_b} 
	\leq 3 \|\cdot\|_{(B(\RR), Lip)_{\alpha,\infty}} 
	\leq  6 \|\cdot\|_{C^\alpha_b}
\end{equation}
(see  Lemma \ref{lemma:holder_interpolationA} and also \cite[Theorem 2.7.2/1]{triebel} in a slightly different 
setting).

\begin{thm} \label{th:Holder_fractional}
		Let $0<\alpha \leq \theta < 1$.
		\begin{enumerate}
			\item[(a)] If $f\in C^\alpha_b$ and $m_{2\alpha/\theta} < \infty$, 
				then
				\[
					f(X_1) \in \left( L_2(\PP), \DD  \right)_{\theta, \infty}
				\]
				and
				\[
					\|f(X_1)\|_{\left( L_2(\PP), \DD  \right)_{\theta, \infty}}\leq  18\sqrt{ 1 + 4 m_{2\alpha/\theta} } 
					 \|f\|_{C^\alpha_b}.
				\]
			\item[(b)] Assume that \ref{A.pos2} holds and choose $t_0\in(0,1)$ and $\ell\in\{0,1,2,\ldots\}$ such that there exist 
				$k\in\mathbbm{Z}$ and $c>0$ with $p_t(x)\geq c$ 
				for all $t\in[t_0,1]$ and all $x\in[(k-1)2^{-\ell},(k+2)2^{-\ell}]$. 
				For the function $g^{\alpha,\ell}\in C_b^\alpha$ of Lemma \ref{lem:ciesielski_function} it holds that
				\[
					g^{\alpha,\ell}(X_1) \in \left( L_2(\PP), \DD  \right)_{\theta, \infty}\quad \text{only if} 
					\quad m_{2\alpha/\theta+\varepsilon}<\infty \text{ for all }\varepsilon >0.
				\] 
	\end{enumerate}
\end{thm}

\begin{proof}
	(a) One finds for every $t>0$ and $\varepsilon>0$ a function $f_t\in C_b^{\alpha/\theta}$ such that
	\[ 
		\left( \|f-f_t\|_\infty + t\|f_t\|_{C_b^{\alpha/\theta}} \right) \leq K(f,t;B(\RR),C_b^{\alpha/\theta}) 
		+ \varepsilon.
	\]
	Using inequality \eqref{eq:D12Calphanorms} for $f_t(X_1)$ we get
	\begin{align*}
		K(f(X_1),t; L_2(\PP),\DD)  
		  & \leq \|(f-f_t)(X_1)\|_{L_2(\PP)} + t \|f_t(X_1)\|_{\DD}                                       \\
		  & \leq \|f-f_t\|_\infty +  t \|f_t\|_{C_b^{\alpha/\theta} }  \sqrt{ 1 + 4 m_{2\alpha/\theta} }    \\*
		  & \leq \sqrt{ 1 + 4 m_{2\alpha/\theta} }\left( K(f,t; B(\RR), C_b^{\alpha/\theta}) + \varepsilon \right) 
	\end{align*}
	so that 
	\[
		\|f(X_1)\|_{(L_2(\PP),\DD)_{\theta,\infty}} 
		\leq \sqrt{ 1 + 4 m_{2\alpha/\theta} } \|f\|_{(B(\RR),C_b^{\alpha/\theta})_{\theta,\infty}}.
	\]
	Using the first inequality of \eqref{eq:Holder_interpolation}, \eqref{eq:reiteration_norms}, and the second inequality of \eqref{eq:Holder_interpolation},  
	we obtain that
	\begin{align*}
		\|f\|_{(B(\RR),C_b^{\alpha/\theta})_{\theta,\infty}} 
		  & \leq 3\|f\|_{(B(\RR),(B(\RR),Lip)_{\alpha/\theta, \infty})_{\theta,\infty}}       \\
		  & \leq  9\|f\|_{(B(\RR),Lip)_{\alpha,\infty}}                                 \\
		  & \leq  18 \|f\|_{C^\alpha_b} 
	\end{align*}
	and this finishes the proof of (a). 
	The proof of assertion (b) is given in Section \ref{section:fractional_smoothness}.
\end{proof}

\begin{rmk}
	Assertion (a) of Theorem \ref{th:Holder_fractional} implies that  $f(X_1) \in (L_2(\PP),\DD)_{\alpha,\infty}$ for all $f\in C_b^\alpha$
	for any pure jump L\'evy process $X$.
	Also for the Brownian motion $B$ we  obtain the smoothness of level $(\alpha,\infty)$ for $f(B_1)$ for any $f\in C^\alpha_b$:
	choose $f_t \in C^1_b = Lip$ like in the proof of Theorem \ref{th:Holder_fractional} and use the fact that 
	\[
		\|f_t(B_1)\|_{\DD} \leq c \|f_t\|_{Lip}
	\] 
	from \cite[Lemma A.5]{Toivola1}, where $c>0$ is a constant not depending on $f_t$.
\end{rmk}


\section{Functions of bounded variation and smoothness}
\label{section:BV}


Let us first recall the space of \emph{normalized functions of bounded variation}, the space $NBV$.
The variation function of $f$ is given by
\[
	T_f(x) = \sup \left\{ \sum_{i=1}^n \left| f(x_i) - f(x_{i-1}) \right| : -\infty < x_0 < x_1 < \cdots < x_n = x,\ n\geq1 \right\}
\]
and the total variation of $f$ is $V(f) = \lim_{x\to\infty} T_f(x)$. 
The space of functions of bounded variation is
\[ 
	BV = \left\{ f:\mathbbm{R} \to \mathbbm{R} : \|f\|_{BV} = \limsup_{x\to-\infty} \left| f(x) \right| + V(f) < \infty \right\}.
\]
Note that when $V(f)< \infty$, then the limit $f(-\infty) := \lim_{x\to-\infty}f(x)$ exists (\cite[Theorem 3.27(c)]{folland}) and for 
$f\in BV$ we may write
$\|f\|_{BV} = |f(-\infty)| + V(f)$. Furthermore, 
\begin{equation*} 
	\|f\|_{\infty} \leq \|f\|_{BV}.
\end{equation*}
We denote by $NBV$ the space of normalized functions of bounded variation, that is, the space of all $f\in BV$ such that $f$ is right continuous
and $f(-\infty)=0$. When $f\in NBV$, then by \cite[Theorem 3.29]{folland} there exists a finite signed measure $\mu_f$ such that 
	\begin{equation} \label{eq:mu_f}
		f(x) =  \int_\RR \one_{(-\infty,x]}(u)\mu_f(\ud u)  = \int_\RR \one_{[u,\infty)}(x)\mu_f(\ud u) = \int_\RR \one_{[0,\infty)}(x-u) \mu_f(\ud u)
	\end{equation}
	for all $x\in\RR.$ Furthermore, $\mu_f$ admits the Jordan decomposition $\mu_f = \mu_f^+ - \mu_f^-$, where $\mu_f^+$ and $\mu_f^-$ are nonnegative  finite measures. 
	We write $|\mu_f| = \mu_f^+  +\mu_f^-$ so that $|\mu_f|(\RR) = \|f\|_{BV}$.


\subsection{Smoothness of first order}

\begin{thm}[{\cite[Example 3.1]{laukkarinen2012}}]\label{th:BV_gives_DD}
	For normalized functions of bounded variation we have the following.
	\begin{enumerate} 
		\item[(a)] Assume that \ref{A.bdd} holds. If $f\in NBV$ and $m_1 < \infty$, then $f(X_1)\in\DD$  and
		      \[
			      \|f(X_1)\|_{\DD} \leq  \sqrt{1+(1\vee \|p_1\|_\infty) m_1}\|f\|_{BV}.
		      \]
		\item[(b)] Suppose that $X_1$ satisfies \ref{A.pos} and let $K\in\RR$ be such that there is $r>0$ and $c>0$ such that the density $p_1$ of $X_1$ satisfies
		       $p_1(x)\geq c$ for all $x\in[K-r,K+r]$. Then 
		      $\one_{[K,\infty)}(X_1)\in\DD$ only if $m_1 < \infty$.
	\end{enumerate}
\end{thm}
\begin{proof}
	(a) Let $f\in NBV$ and $\mu_f$ be the according signed measure from \eqref{eq:mu_f}. We use H\"older's inequality to get
	\begin{align*}
		&\int_\RR \EE\left[ \left( f(X_1 + x) -f(X_1)  \right)^2 \right] \nu(\ud x)\\
		& = \int_\RR \EE\left[ \left( \int_\RR \left( \one_{[u,\infty)}(X_1 + x) - \one_{[u,\infty)}(X_1) \right) \mu_f(\ud u)  \right)^2 \right] \nu(\ud x) \\*
		& \leq |\mu_f|(\RR) \int_\RR  \int_\RR \EE \left[ \left( \one_{[u,\infty)}(X_1 + x) - \one_{[u,\infty)}(X_1) \right)^2 \right] |\mu_f|(\ud u)  \nu(\ud x) \\
		& \leq |\mu_f|(\RR) \int_\RR  \int_\RR \left(  \|p_1\|_\infty |x| \wedge 1  \right) |\mu_f|(\ud u)  \nu(\ud x) \\
		& \leq \|f\|_{BV}^2 (1\vee\|p_1\|_\infty ) \int_\RR (|x| \wedge 1) \nu(\ud x).
	\end{align*}
	Hence from \eqref{eq:quotient_in_L_2} we obtain that
	\begin{align*}
		\|f(X_1)\|^2_{\DD}& = \|f(X_1)\|^2_{L_2(\PP)} + \int_\RR \EE\left[ \left( f(X_1 + x) -f(X_1)  \right)^2 \right] \nu(\ud x) \\*
		                  & \leq \|f\|_{BV}^2 + \|f\|_{BV}^2 (1\vee\|p_1\|_\infty) m_1.
	\end{align*}
	(b) Let $r>0$ and $c>0$ be such that $p_1(x) \geq c$ for all $x\in[K-r,K+r]$. 
	Let $f=\one_{[K,\infty)}$. Then $f\in NBV$ and
	\begin{align*}
		&  \int_\RR \EE\left[ \left| f(X_1 + x) -f(X_1)  \right|^2 \right] \nu(\ud x)\\
		& = \int_{(-\infty,0)} \EE \left[ \one_{[K,K-x)}(X_1) \right] \nu(\ud x) 
		+ \int_{(0,\infty)}\EE \left[ \one_{[K-x,K)}(X_1) \right] \nu(\ud x)\\
		& \geq c\int_{0<|x|\leq r} |x|\nu(\ud x).
	\end{align*}
	By \eqref{eq:quotient_in_L_2} it holds that $m_1<\infty$, if $f(X_1)\in\DD$.
\end{proof}


\subsection{Fractional smoothness}

If $m_1<\infty$ does not hold, it is still possible to attain fractional smoothness with functions in $NBV$. In 
\cite[Example 4.2(a)]{geiss-geiss-laukkarinen} it is verified that $\one_{(K,\infty)}(X_1) \in 
(L_2(\PP),\DD)_{1/2,\infty}$. Note that in \cite[Example 4.2(a)]{geiss-geiss-laukkarinen} it is assumed 
a small ball estimate for the distribution and this assumption is equivalent with \ref{A.bdd} (one can easily 
see this by using the steps of the proof of \cite[Theorem 2.4(iii)]{avikainen1}). 
In the following theorem we show that the smoothness level increases as the Blumenthal-Getoor index decreases.

\begin{thm}\label{th:BV_fractional}
	Let $1/2 \leq \theta < 1$. 
	\begin{enumerate} 
	 \item[(a)] 
		Assume that \ref{A.bdd} holds. If $f\in NBV$ and $m_{1/\theta} < \infty$, then 
		\[
			f(X_1) \in \left( L_2(\PP), \DD  \right)_{\theta,\infty}
		\]
		and
		\[ 
			\|f(X_1)\|_{(L_2(\PP),\DD)_{\theta,\infty}} \leq \left( \sqrt{\|p\|_\infty} 
			+ \sqrt{1+2\left( \|p\|_\infty \vee 1 \right) m_{1/\theta}} \right)   \|f\|_{BV}.
		\]
		Especially, $f(X_1) \in \left( L_2(\PP), \DD  \right)_{\frac{1}{2},\infty}$ for any L\'evy measure $\nu$.
	\item[(b)]
		Assume that \ref{A.pos2} holds and let $t_0\in(0,1)$ and $K\in\RR$ be such that there exist $r>0$ and $c>0$ with
		$p_t(x)\geq c$ for all $x\in[K-2r,K+2r]$ and all $t\in[t_0,1]$. Then 
		\[
			\one_{[K,\infty)}(X_1) \in \left( L_2(\PP), \DD  \right)_{\theta,\infty}\quad \text{only if} \quad m_{1/\theta + \varepsilon}< \infty \text{ for all }\varepsilon>0.
		\]
		
	\end{enumerate}
\end{thm}
\begin{proof}
	(a) Let $f\in NBV$ and $\mu_f$ be the according signed measure from \eqref{eq:mu_f}.
	For $t\in(0,1)$ we define
	\[
		g_t(x) =  
		\begin{cases}
			0, & x\leq 0\\
			\frac{1}{t}x^{\frac1{2\theta}}, & 0< x < t^{2\theta} \\
			1, & x\geq t^{2\theta}
		\end{cases}
		\quad \text{ and } \quad f_t(x) = \int_\RR g_t(x-u) \mu_f(\ud u).
	\]
	Then 
	\begin{align*}
		& \EE \left[ \left( f_t(X_1+x) - f_t(X_1)  \right)^2 \right]                                             \\
		& =  \int_\RR \left( \int_\RR \left[ g_t(y+x-u) - g_t(y-u) \right] \mu_f(\ud u)  \right)^2 p(y)\ud y          \\
		& \leq  |\mu_f|(\RR) \|p\|_\infty  \int_\RR\int_\RR \left( g_t(y+x-u) - g_t(y-u) \right)^2 |\mu_f|(\ud u)\ud y \\
		& = |\mu_f|(\RR)^2 \|p\|_\infty  \int_\RR \left( g_t(z+x) - g_t(z) \right)^2 \ud z   .                                  
	\end{align*}
	Note that $g_t(\cdot+x)-g_t$ is nonzero only on an interval of length $t^{2\theta}+|x|$ and
	\begin{align*}
		|g_t(z+x)-g_t(z)|
		&  = \left| \int_z^{z+x} \frac1{2\theta t} u^{\frac1{2\theta}-1}\one_{(0,t^{2\theta})}(u) \ud u  \right| \\
		&  \leq \int_0^{|x|} \frac1{2\theta t} u^{\frac1{2\theta}-1}\one_{(0,t^{2\theta})}(u) \ud u \\
		& = g_t(|x|) \leq 1 
	\end{align*} 
	for all $x,z\in\RR$,  since $\frac1{2\theta}-1 \leq 0$. When $|x| \geq t^{2\theta}$, then
	\[
		\int_\RR\left( g_t(z+x) - g_t(z) \right)^2 \ud z 
		\leq 2|x| = 2 t^{2(\theta -1)} |x| t^{2(1-\theta)} 
		\leq 2  t^{2(\theta -1)} |x|^{1/\theta}.
	\]
	When $|x| < t^{2\theta}$, then 
	\[
		\int_\RR \left( g_t(z+x) - g_t(z) \right)^2 \ud z  \leq 2t^{2\theta} g_t^2(|x|) 
		= 2 t^{2(\theta-1)} |x|^{1/\theta}.
	\]
	On the other hand,
	\begin{align*}
		& \EE \left[ \left( f_t(X_1+x) - f_t(X_1)  \right)^2 \right] \\
		& = \EE \left[ \left( \int_\RR \left( g_t(X_1+x-u) - g_t(X_1-u) \right) \mu_f(\ud u) \right)^{\!\!2}  \right] \\
		&\leq |\mu_f|^2(\RR),
	\end{align*}
	so that
	\begin{align*}     
		& \int_\RR \EE\left[ \left( f_t(X_1 + x) -f_t(X_1)  \right)^2 \right] \nu(\ud x) \\
		& \leq  \int_\RR |\mu_f|(\RR)^2 \left( \|p\|_\infty \vee 1 \right) \left( 2 t^{2(\theta -1)}|x|^{1/\theta} \wedge 1 \right) \nu(\ud x) \nonumber \\
		& \leq |\mu_f|(\RR)^2 \left( \|p\|_\infty \vee 1 \right) 2 t^{2(\theta -1)} m_{1/\theta}                                                        
	\end{align*}
	since $0<t<1$, and therefore $f_t(X_1) \in\DD$. It also holds,  by \eqref{eq:mu_f}, that
	\begin{align*}
		\|(f-f_t)(X_1)\|^2_{L_2(\PP)}
		& = \int_\RR \left( \int_\RR \left[ \one_{[0,\infty)}(y-u) - g_t(y-u) \right] \mu_f(\ud u) \right)^2 \PP_{X_1} (\ud y)\nonumber \\
		& \leq  |\mu_f|(\RR)^2 \|p\|_\infty\int_\RR \left( \one_{[0,\infty)}(y) - g_t(y) \right)^2 \ud y\nonumber                       \\ 
		& \leq  |\mu_f|(\RR)^2 \|p\|_\infty t^{2\theta}                                                                                
	\end{align*}
	and 
	\[
		\|f_t(X_1)\|_{L_2(\PP)} \leq |\mu_f|(\RR).
	\]
	We obtain for $t\in(0,1)$ that
	\begin{align*}
		& t^{-\theta}\left( \|(f-f_t)(X_1)\|_{L_2(\PP)} + t\sqrt{ \|f_t(X_1)\|^2_{L_2(\PP)}                     
		+ \|D f_t(X_1)\|^2_{L_2(\mm\otimes\PP)} }  \right) \\*
		  & \leq t^{-\theta}\left(  \sqrt{\|p\|_\infty}|\mu_f|(\RR) t^{\theta} + t  \sqrt{ |\mu_f|(\RR)^2           
		+ |\mu_f|(\RR)^2 \left( \|p\|_\infty \vee 1 \right) 2 t^{2(\theta -1)} m_{1/\theta}  } \right) \\*
		  & \leq  \left( \sqrt{\|p\|_\infty} + \sqrt{1+2\left( \|p\|_\infty \vee 1 \right)  m_{1/\theta}} \right) 
		|\mu_f|(\RR). 
	\end{align*}
	Thus
	\begin{align*}
		& \left\| f(X_1) \right\|_{(L_2(\PP),\DD)_{\theta,\infty}}                                                \\
		& = \sup_{t>0} t^{-\theta} \inf \{ \|Y_0\|_{L_2(\PP)} + t\|Y_1\|_{\DD}: Y_0 + Y_1 = f(X_1) \}             \\
		& \leq \sup_{t\in(0,1)} t^{-\theta}\left( \|(f-f_t)(X_1)\|_{L_2(\PP)} + t\sqrt{ \|f_t(X_1)\|^2_{L_2(\PP)} 
		+ \|D f_t(X_1)\|^2_{L_2(\mm\otimes\PP)} }  \right) \\
		& \qquad \vee \|f(X_1)\|_{L_2(\PP)}                                                                       \\
		& \leq \left( \sqrt{\|p\|_\infty} + \sqrt{1+2\left( \|p\|_\infty \vee 1 \right)  m_{1/\theta}} \right)    
		\|f\|_{BV}.
	\end{align*}
	The proof of assertion (b) is given in Section \ref{section:fractional_smoothness}.
\end{proof}
%


\section{Sharpness of the connection between the smoothness index and the Blumenthal-Getoor index}
\label{section:fractional_smoothness}


In Lemma \ref{lemma:characterization_for_fractional_smoothness} below, we adapt the characterisation for fractional smoothness 
from \cite[Corollary 2.3]{geiss-hujo}, where it is written for the Brownian motion.

\begin{definition}
	For a sequence of Banach spaces $E = (E_n)_{n=0}^\infty$ with $E_n \neq \{0\}$
	we let $\ell_2 (E)$ and $d_{1,2} (E)$ be the Banach spaces of all $a = (a_n)_{n=0}^\infty \in E$ such
	that
	\[
		\|a\|_{\ell_2 (E)} := \left(\sum_{n=0}^\infty \|a_n\|^2_{E_n}\right)^\frac{1}{2}
		\quad \text{and} \quad \|a\|_{d_{1,2}(E)}
		:= \left(\sum_{n=0}^\infty (n + 1) \|a_n\|^2_{E_n}\right)^\frac{1}{2}
	\]
	respectively, are finite. For $a\in E$ we let $Ta:[0,1]\to\RR$ be defined by
	\[
		(Ta)(t) := \sum_{n=0}^\infty \|a_n\|^2_{E_n} t^n.
	\]
\end{definition}
 We use the notation $A\sim_c B$ for $\tfrac1c B \leq A \leq cB$, where $A,B \in[0,\infty]$ and $c\geq 1$.
\begin{lem}[{\cite[Theorem 2.2]{geiss-hujo}}] \label{lemma:thm_geisshujo}
	For $\theta\in(0,1)$, $q\in[1,\infty]$ and $a \in \ell_2(E)$ one has
	\begin{align*}
		  & \|a\|_{ (\ell_2(E),d_{1,2}(E))_{\theta,q} }  \\
		  &  \sim_c \|a\|_{\ell_2 (E)} +               
		\left\| (1-t)^{ \frac{1-\theta}{2} } \sqrt{ (Ta)'(t) }  \right\|_{L_q\left((0,1),\bor(0,1),\frac{\ud t}{1-t}\right) }  \\
		  & \sim_c \|a\|_{\ell_2 (E)} +                 
		\left\| (1-t)^{-\frac{\theta}{2}} \sqrt{ (Ta)(1) - (Ta)(t) }  \right\|_{L_q\left((0,1),\bor(0,1),\frac{\ud t}{1-t}
		\right)},
	\end{align*}
	where $c\geq 1$ depends only on $(\theta,q)$, and the expressions may be infinite.
\end{lem}
We will apply this theorem to the It\^{o} chaos decomposition.  Let $(\ftn_t)_{t\geq 0}$ be 
	the augmented natural filtration of $X$.
Throughout this section we let $\bar{X}$ be an independent copy of $X$ on $(\bar{\Omega},\bar{\ftn}, \bar{\PP})$.
We will use the notation $\bar{\EE}$ for the expectation with respect to the measure $\bar{\PP}$.

\begin{lem}\label{lemma:characterization_for_fractional_smoothness}
	For $\theta\in(0,1)$, $q\in[1,\infty]$ and $f(X_1) \in L_2(\PP)$ one has
	\begin{align*}
		  & \|f(X_1)\|_{ (L_2(\PP),\DD)_{\theta,q} } \\
		  & \sim_c \|f(X_1)\|_{L_2 (\PP)} +          
		\left\| (1-t)^{-\frac{\theta}{2}} \! \left\| f(X_1) - \EE \left[ f(X_1) | \ftn_t  \right] 
		\right\|_{L_2(\PP)}  \right\|_{L_q\left((0,1),\bor(0,1),\frac{\ud t}{1-t}\right)} \\
		  & =  \|f(X_1)\|_{L_2 (\PP)} \!  + \! \tfrac1{\sqrt2}\left\| (1-t)^{-\frac{\theta}{2}} \!\left\| \left\| f(X_1) - f(X_t+\bar{X}_{1-t}) 
		\right\|_{L_2(\PP)} \right\|_{L_2(\bar{\PP})}  \right\|_{L_q\left(\! \frac{\ud t}{1-t} \!\right)}\!\!,
	\end{align*}
	where $c\geq 1$ depends only on $(\theta,q)$ and the expressions may be infinite. 
\end{lem}
\begin{proof}
	Let $f(X_1) = \sum_{n=0}^\infty I_n(f_n) \in L_2(\PP)$, $E = \left( L_2(\mm^{\otimes n}) \right)_{n=0}^\infty$ 
	and $a=\left( \sqrt{n!} \tilde{f_n} \right)_{n=0}^\infty$. By orthogonality the equality 
	\[
		\sum_{n=0}^\infty I_n(f_n) = \sum_{n=0}^\infty I_n(g_n) + \sum_{n=0}^\infty I_n(h_n)
	\] 
	holds in $L_2(\PP)$ if and only if $\tilde{f_n} = \tilde{g_n} + 
	\tilde{h_n}$ holds $\mm^{\otimes n}$-a.e. Therefore
	\begin{align*}
		  & K(f(X_1),t;L_2(\PP),\DD)                                                                            \\*
		  & = \inf_{ \tilde{f_n} = \tilde{g_n} + \tilde{h_n} } \left( \sqrt{\sum_{n=0}^\infty n! \| \tilde{g_n} 
		\|^2_{ L_2( \mm^{\otimes n} ) }} 
		+ t  \sqrt{\sum_{n=0}^\infty (n+1)! \| \tilde{h_n} \|^2_{ L_2( \mm^{\otimes n} ) }}  \right)\\*
		  & = K(a,t;\ell_2(E),d_{1,2}(E))                                                                     
	\end{align*} 
	and Lemma \ref{lemma:chaos_martingale}(b) gives
	\begin{align*}
		\left\| f(X_1) - \EE \left[ f(X_1) | \ftn_t  \right] \right\|_{L_2(\PP)}^2
		  & = \EE \left[ f(X_1)^2 \right]  - \EE \left[ \EE \left[ f(X_1) | \ftn_t  \right]^2 \right] \\*
		  & = (Ta)(1) - (Ta)(t).                                                                   
	\end{align*}
	The equivalence follows now from Lemma \ref{lemma:thm_geisshujo}. To conclude with the equality below, we use the facts that
	$\EE \left[ f(X_1) | \ftn_t  \right] = \bar{\EE} \left[  f(X_t+ \bar{X}_{1-t}) \right]$ a.s. and
	$X_t + \bar{X}_{1-t} \stackrel{d}{=} X_1$ to get that
	\begin{align*}
		& \left\| f(X_1) - \EE \left[ f(X_1) | \ftn_t  \right] \right\|_{L_2(\PP)}^2\\
		  & = \EE \left[ f(X_1) \big( f(X_1) -  \EE \left[ f(X_1) | \ftn_t  \right] \big) \right] \\
		  & = \bar{\EE}\EE \left[ f(X_1) \big( f(X_1) -  f(X_t+ \bar{X}_{1-t})  \big) \right] \\
		  & = -\bar{\EE}\EE \left[ f(X_t+\bar{X}_{1-t}) \big( f(X_1) -  f(X_t+ \bar{X}_{1-t})  \big) \right] \\
		  & = \frac12\bar{\EE}\EE \left[ \big( f(X_1) -  f(X_t+ \bar{X}_{1-t})  \big)^2 \right],
	\end{align*}
	where the last line is obtained as the average of the two previous lines.
\end{proof}
\begin{lem}\label{lem:index_bound}
	Let $\tilde{X}$ be a {\pink pure jump} L\'evy process with c\`{a}dl\`{a}g paths on some probability space 
	$(\tilde{\Omega},\tilde{\cal F},\tilde{\mathbbm{P}})$. {\pink Let $\tilde{\nu}$ be its L\'evy measure
	and}  $\beta$ be its Blumenthal-Getoor index. 
	Let $t_0>0$  {\pink and define a constant $\kappa$ by letting
	\[
		\kappa = \begin{cases}
		          \int_{\{|x|\leq1\}} x \tilde{\nu}(\ud x), & \text{if }\int_{\{|x|\leq1\}} |x| \tilde{\nu}(\ud x)<\infty\\
		          0, & \text{if }\int_{\{|x|\leq1\}} |x| \tilde{\nu}(\ud x)=\infty.
		         \end{cases}
	\]
	}
	\begin{enumerate} 
	 \item[(a)] For all $\beta'>\beta$ it holds that
		\[
			\int_0^{t_0} \tilde{\PP}\left(\frac{|\tilde{X}_t{\pink + \kappa t}|}{t^{1/\beta'}} > c\right) \frac{\ud t}{t} < \infty \quad \text{ for all } c>0.
		\]
	\item[(b)]  For any $\beta''< \beta$ there exists $c'>0$ such that
		\[
			\int_0^{t_0} \tilde{\PP}\left(\frac{|\tilde{X}_t{\pink + \kappa t}|}{t^{1/\beta''}} > c'\right) \frac{\ud t}{t} = \infty.
		\]
	\item[\pink (c)] \pink It holds that \[
			\int_0^{t_0} \tilde{\PP}\left(|\tilde{X}_t| > c\right) \frac{\ud t}{t} < \infty \quad \text{ for all } c>0.
		   \]
	\end{enumerate}	
\end{lem}
\begin{proof}
	By \cite[Theorems 3.1 and 3.3]{blumenthal-getoor} it holds  {\pink for all $\beta'' < \beta < \beta'$} that 
	\[
	 \lim_{t\to0} \frac{\tilde{X}_t {\pink + \kappa t}}{t^{1/\beta'}} = 0 \text{ a.s. \,\, and \,\,} \limsup_{t\to0} \frac{|\tilde{X}_t{\pink + \kappa t}|}{t^{1/\beta''}} = \infty \text{ a.s.}
	\]
	By a result of Khintchine \cite[Section 2]{khintchine1939}, we have that if $u:(0,t_0)\to(0,\infty)$ is non-decreasing {\pink and 
	$\displaystyle \lim_{t\to 0}u(t) = 0$, then for any L\'evy process $Y$ it holds that }
	\[
		\lim_{t\to0}\frac{{\pink Y_t}}{u(t)} = 0 \text{ a.s. }\quad\text{if and only if}\quad \int_0^{t_0} \PP\left(\frac{|{\pink Y_t}|}{u(t)}>c\right)\frac{\ud t}{t}<\infty \text{ for all }c>0.
	\]
	The claims {\pink (a) and (b)} follow by choosing {\pink $Y_t = \tilde{X}_t + \kappa t$ and} $u(t)=t^{1/\beta'}$ in (a) and $u(t)=t^{1/\beta''}$ in (b). \\
	{\pink (c) 
	Let $u(t) = t^{1/3}\wedge 1$. Then 
	\[
		\lim_{t\to0} \frac{\tilde{X}_t}{u(t)} \leq \lim_{t\to0} \frac{|\tilde{X}_t+ {\pink \kappa } t|}{t^{1/3}} + \frac{|{\pink \kappa } t|}{t^{1/3}} = 0,
	\]
	by (a) so that \cite[Section 2]{khintchine1939} implies that
	\[
		 \int_0^{t_0} \tilde{\PP}\left(|\tilde{X}_t| > c\right) \frac{\ud t}{t}
		 \leq \int_0^{t_0} \tilde{\PP}\left( \frac{|\tilde{X}_t|}{u(t)} > c\right) \frac{\ud t}{t}< \infty \quad \text{ for all } c>0.
	\]}
\end{proof}

\begin{lem} \label{lem:f_index}
	Assume that \ref{A.pos2} holds and let $R>0$, $a<b$, $t_0\in(0,1)$ and $c_1>0$ be such that $p_t(x)\geq c_1$ for all $x\in [a-R,b+R]$ and $t\in[t_0,1]$. 
	If $f:\RR\to\RR$ is Borel measurable and
	there exist $r>0$, $c_2>0$ and $\eta>0$ such that 
	\begin{equation} \label{eq:f_assumption}
		\int_a^b\left[ f(y+x) - f(y) \right]^2 \ud y \geq c_2|x|^\eta  \quad \text{ for all }|x|\leq r,
	\end{equation}
	then
	\[
		f(X_1) \in (L_2(\PP),\DD)_{\theta,\infty} \quad \text{only if} \quad m_{\eta/\theta + \varepsilon} < \infty
	\]
	for all $\varepsilon > 0$. 
\end{lem}

\begin{proof} The assumptions \ref{A.pos2} and \eqref{eq:f_assumption} yield for $t\in[t_0,1]$ that
	\begin{align} \label{align:lower_bound}
		& \left\| \left\| f(X_1) - f(X_t+\bar{X}_{1-t}) \right\|_{L_2(\PP)} \right\|_{L_2(\bar{\PP})}^2 \nonumber \\
		& = \bar{\EE}\EE \left[ \int_\RR \left( f(y+X_1-X_t) - f(y+\bar{X}_{1-t})   \right)^2 p_t(y)\ud y \right] \nonumber  \\
		& = \bar{\EE}\EE \left[ \int_\RR \left( f(y+X_{1-t}-\bar{X}_{1-t}) - f(y)   \right)^2 p_t(y-\bar{X}_{1-t})\ud y \right] \nonumber \\
		& \geq \bar{\EE}\EE \left[ \int_a^b \left( f(y+X_{1-t}-\bar{X}_{1-t}) - f(y)   \right)^2 c_1 \ud y \one_{\{ |\bar{X}_{1-t}|\leq R \}} \right] \nonumber \\
		& \geq c_1c_2 \bar{\EE}\EE \left[ |X_{1-t}-\bar{X}_{1-t}|^{\eta}\one_{\{|X_{1-t}-\bar{X}_{1-t}|\leq r, |\bar{X}_{1-t}|\leq R\}} \right].
	\end{align}
	Since $X$ and $\bar{X}$ are independent, the process $\tilde{X}=X-\bar{X}$ with $\tilde{X}_t(\omega,\bar{\omega}) = X_t(\omega)-\bar{X}_t(\bar{\omega})$ 
	is a L\'evy process on $(\Omega\times\bar{\Omega}, \ftn\otimes\bar{\ftn}, \PP\otimes\bar{\PP})$ 
	with L\'evy measure $\tilde{\nu}(B) = \nu(B)+\nu(-B)$, and its 
        Blumenthal-Getoor index is the same $\beta$ as for $X$. 
	Let $0<\theta' < \theta$ and $c>0$ and set $c_3 = c_1c_2c^\eta$. Then \eqref{align:lower_bound} has the lower bound
	\begin{align*}
		&  c_3 (1-t)^{\theta'}(\PP\otimes \bar{\PP})\left( \frac{|X_{1-t}-\bar{X}_{1-t}|^{\eta}}{(1-t)^{\theta'}c^\eta} > 1, 
		                                             |X_{1-t}-\bar{X}_{1-t}|\leq r, |\bar{X}_{1-t}|\leq R\right)\\
		& \geq c_3(1-t)^{\theta'}\left[  \tilde\PP \left( \frac{|\tilde{X}_{1-t}|}{(1-t)^{\theta'/\eta}} > c \right) 
		                             - \tilde\PP\left( |\tilde{X}_{1-t}|> r\right) - \bar\PP(|\bar{X}_{1-t}| > R) \right].
	\end{align*}
	 If $f(X_1)\in(L_2(\PP),\DD)_{\theta,\infty}$, then $f(X_1)\in(L_2(\PP),\DD)_{\theta',2}$ by \eqref{eq:lexicographical_order}. 
	 Using Lemma \ref{lemma:thm_geisshujo} we get that 
	\begin{align*}  
	 \infty
	 & > \int_0^1 (1-t)^{-\theta'} \left\| \left\| f(X_1) - f(X_t+\bar{X}_{1-t}) \right\|_{L_2(\PP)} \right\|_{L_2(\bar{\PP})}^2 \frac{\ud t}{1-t}\\
	 & \geq c_3\int_{{\pink 1-t_0}}^1 \left[  \tilde{\PP} \left( \frac{|\tilde{X}_{1-t}|}{(1-t)^{\theta'/\eta}} > c \right) 
	                                      - \tilde\PP\left( |\tilde{X}_{1-t}|> r\right) - \bar\PP(|\bar{X}_{1-t}| > R) \right] \frac{\ud t}{1-t}\\
	 & = c_3 \left[\int_0^{t_0} \tilde{\PP} \left( \frac{|\tilde{X}_t|}{t^{\theta'/\eta}} > c\right) \frac{\ud t}{t} 
	     - \int_0^{t_0} \tilde\PP(|\tilde{X}_t|>r) \frac{\ud t}{t} - \int_0^{t_0} \bar\PP(|\bar{X}_t|>R) \frac{\ud t}{t} \right],
	\end{align*}
	where {\pink
	\[
		\int_0^{t_0} \tilde\PP(|\tilde{X}_t|>r) \frac{\ud t}{t} 
		+ \int_0^{t_0} \bar\PP(|\bar{X}_t|>R) \frac{\ud t}{t}
		< \infty
	\]
        by Lemma \ref{lem:index_bound}(c).} 
        Hence
        \[
         \int_0^{t_0} \tilde{\PP} \left( \frac{|\tilde{X}_t|}{t^{\theta'/\eta}} > c\right) \frac{\ud t}{t}  < \infty \quad \text{for all }c>0 \text{ and for all }0<\theta'<\theta.
        \]
	{\pink Since $\tilde{\nu}$ is symmetric, the constant $\kappa$ of Lemma \ref{lem:index_bound} is zero and} Lemma \ref{lem:index_bound}(b) 
	implies $\beta \leq \eta/\theta'$ for all $0<\theta'<\theta$, 
	so that $\beta\leq \eta/\theta$.
\end{proof}

\begin{proof}[Proof of Theorem \ref{th:Holder_fractional}(b)]
	By Lemma \ref{lem:ciesielski_function}, the function $g^{\alpha,\ell}$ satisfies \eqref{eq:f_assumption}, when we choose $[a,b] = [k2^{-\ell},(k+1)2^{-\ell}]$, 
	$r=2^{-\ell-3}$, $c_2=2^{-\ell}2^{8\alpha-10}$ and $\eta = 2\alpha$.
	If $g^{\alpha,\ell}(X_1)\in(L_2(\PP),\DD)_{\theta,\infty}$, then by Lemma \ref{lem:f_index} it holds that $\beta\leq2\alpha/\theta$. 
\end{proof}
\begin{proof}[Proof of Theorem \ref{th:BV_fractional}(b)]
	We have that
	\begin{align} \label{align:K_lower}
		&  \int_{K-r}^{K+r} \left( \one_{[K,\infty)}(y+x) - \one_{[K,\infty)}(y)  \right)^2 \ud y \nonumber\\
		& = \int_{K-r}^{K+r} \left( \one_{[K-x,K)}(y)\one_{(0,\infty)}(x) + \one_{[K,K-x)}(y)\one_{(-\infty,0)}(x)  \right) \ud y \nonumber\\
		& = |x|
	\end{align}
        for all $|x|\leq r$, so that $\one_{[K,\infty)}$ satisfies \eqref{eq:f_assumption} with $[a,b]=[K-r, K + r]$.
	Choosing $R=r$ it now follows from Lemma \ref{lem:f_index}, that if $\one_{[K,\infty)}(X_1)\in(L_2(\PP),\DD)_{\theta,\infty}$
	then $\beta\leq 1/\theta$. 
\end{proof}
\begin{rmk}
	\begin{enumerate}
	\item[(a)]If $m_\beta < \infty$ and \ref{A.pos2} holds, then we get for $0<\alpha\leq\theta<1$ from Theorem \ref{th:Holder_fractional}  
	the ''if and only if''-condition
	\[
		m_{2\alpha/\theta}<\infty \, \Longleftrightarrow \, f(X_1)\in (L_2(\PP),\DD)_{\theta,\infty} \,\forall f\in C_b^\alpha
	\] 
	and if also \ref{A.bdd} holds, then Theorem \ref{th:BV_fractional} implies for $1/2 \leq \theta < 1$ that
	\[
		m_{1/\theta}<\infty \, \Longleftrightarrow \, f(X_1)\in (L_2(\PP),\DD)_{\theta,\infty} \,\forall f\in NBV.
	\] 
	Note that $m_\beta<\infty$ is indeed possible:
	choose for example
	\[
		\nu(\ud x) = \frac{b}{|x|^{1+\beta}(\log^2 x+1)}\ud x \quad \text{ for some }b>0
	\]
	for $\beta\in(0,2]$. Using Lemma \ref{lem:assumptions} we see that this process satisfies \ref{A.bdd}-\ref{A.pos2}.
	\item[(b)] If $m_\beta =\infty$, then Theorems \ref{th:Holder_fractional} and \ref{th:BV_fractional} 
	do not give an ''if and only if''-result in general: 
	In Theorems \ref{th:stable_H}-\ref{th:stable_K} in Section \ref{section:stable} we consider the symmetric strictly stable process with
	\[
		\nu(\ud x) = \frac{b}{|x|^{1+\beta}}\ud x \quad \text{ for some }b>0 \text{ and }\beta\in(0,1),
	\]
	and the process satisfies \ref{A.bdd}-\ref{A.pos2} by Lemma \ref{lem:assumptions}. Theorems \ref{th:stable_H} and \ref{th:stable_f} show
	that when $0<\alpha<\theta<1$, then
	\[
		f(X_1)\in (L_2(\PP),\DD)_{\theta,\infty}\, \forall f\in C_b^\alpha \quad \text{ for }2\alpha/\theta = \beta, 
	\] 
	and that for $\tfrac12 \leq \theta < 1$ it holds that
	\[
		f(X_1)\in (L_2(\PP),\DD)_{\theta,\infty}\, \forall f\in NBV \quad \text{ for }1/\theta = \beta, 
	\]
	eventhough $m_\beta = \infty$. However, we obtain for $0<\alpha<\theta<1$ from Theorem \ref{th:stable_H}, that
	\begin{align*}
		  m_{2\alpha/\theta}<\infty \, 
		& \Longleftrightarrow \, f(X_1)\in (L_2(\PP),\DD)_{\theta,q} \,\forall f\in C_b^\alpha \text{ for some }q\in[1,\infty)\\
		& \Longleftrightarrow \, f(X_1)\in (L_2(\PP),\DD)_{\theta,q} \,\forall f\in C_b^\alpha \text{ for all }q\in[1,\infty).
	\end{align*} 
	Theorems \ref{th:stable_f} and \ref{th:stable_K} imply for $0<\theta<1$ that
	\begin{align*}
		  m_{1/\theta}<\infty \, 
		& \Longleftrightarrow \, f(X_1)\in (L_2(\PP),\DD)_{\theta,q} \,\forall f\in NBV \text{ for some }q\in[1,\infty)\\
		& \Longleftrightarrow \, f(X_1)\in (L_2(\PP),\DD)_{\theta,q} \,\forall f\in NBV \text{ for all }q\in[1,\infty).
	\end{align*} 
	\end{enumerate}
\end{rmk}


\subsection{Symmetric strictly stable process}\label{section:stable}

We consider the symmetric strictly stable process which has the characteristic function
$\varphi(u) = e^{-c |u|^\beta}$ for some $c>0$ and $\beta\in(0,2]$ (\cite[Theorem 14.14]{sato}).   If $\beta=2$, then the process is 
the Brownian motion $\sqrt{2c} B$, and otherwise it is a pure jump L\'evy process $X$ with L\'evy measure 
\[
	\nu(\ud x) = b|x|^{-\beta-1}\ud x \quad \text{for some } b>0,
\]
where $\beta$ is the Blumenthal-Getoor index of the process.
We will later take advantage of the property that $X_t \stackrel{d}{=} t^{1/\beta}X_1$, which follows from
\[
 \EE\left[e^{iuX_t}\right] =  e^{- tc|u|^\beta} = e^{-c \left|ut^{1/\beta}\right|^\beta} 
     = \EE\left[ e^{iut^{1/\beta}X_1}\right].
\]
Using  Lemma \ref{lem:assumptions} one can easily check that assumptions \ref{A.bdd}, \ref{A.pos}  and \ref{A.pos2} are satisfied.
For the rest of this section we assume that $X$ is the symmetric and strictly stable process of index $\beta\in(0,2)$.

\begin{lem} \label{lem:stable_lowerBound}
	Let $a<b$ and $t_0\in(0,1)$.
	If $f:\RR\to\RR$ is Borel measurable and
	there exist $r>0$, $c_{\pink 2}>0$ and $\eta>0$ such that \eqref{eq:f_assumption} holds,
	then there exists $c>0$ such that
	\[
		 \left\| \left\| f(X_1) -f(X_t+\bar{X}_{1-t})   \right\|_{L_2(\PP)} \right\|^2_{L_2(\bar{\PP})} \geq c(1-t)^{\eta/\beta} \quad\text{for all }t\in[t_0,1].
	\]	
\end{lem}
\begin{proof}
	Let $R>0$. Since $X_1$ has the support $\RR$ by \cite[Theorem 24.10(ii)]{sato}, then $p_1$ is strictly positive and continuous on $\RR$
	by the proof of Lemma \ref{lem:assumptions}. Hence we find $c_1>0$ such that $p_1(x)\geq c_1$ for all $-|a-R|t_0^{-1/\beta}\leq x \leq |b+R|t_0^{-1\beta}$.
	 Using the fact that $X_t\stackrel{d}{=}t^{1/\beta}X_1$, we obtain for any $x\in[a-R,b+R]$ that
	\[
		p_t(x) =  t^{-1/\beta}p_1(t^{-1/\beta}x)
		\geq  p_1(t^{-1/\beta}x) \geq c_1 
	\] 
	for all $t\in[t_0,1]$. 
	Using \eqref{align:lower_bound} we get that
	\begin{align*} 
		& \left\| \left\| f(X_1) - f(X_t+\bar{X}_{1-t}) \right\|_{L_2(\PP)} \right\|_{L_2(\bar{\PP})}^2 \nonumber \\
		& \geq c_1{\pink c_2} \bar{\EE}\EE \left[ |X_{1-t}-\bar{X}_{1-t}|^{\eta}\one_{\{|X_{1-t}-\bar{X}_{1-t}|\leq r, |\bar{X}_{1-t}|\leq R\}} \right] \\
		& = c_1{\pink c_2} (1-t)^{\eta/\beta}\bar{\EE}\EE \left[ |X_1-\bar{X}_1|^{\eta}\one_{\{|X_1-\bar{X}_1|\leq r(1-t)^{-1/\beta}, |\bar{X}_1|\leq R(1-t)^{-1/\beta}\}} \right] \\
		& \geq c_1{\pink c_2} (1-t)^{\eta/\beta}\bar{\EE}\EE \left[ |X_1-\bar{X}_1|^{\eta}\one_{\{|X_1-\bar{X}_1|\leq r, |\bar{X}_1|\leq R\}} \right]\\
		&\geq c(1-t)^{\eta/\beta}
	\end{align*}
	for some $c>0$, where we used the fact that since $X_1-\bar{X}_1$ is strictly stable with L\'evy measure $2\nu$, it must be that
	$\bar{\EE}\EE \left[ |X_1-\bar{X}_1|\one_{\{|X_1-\bar{X}_1|\leq r, |\bar{X}_1|\leq R\}} \right]$ is strictly positive.
\end{proof}

\begin{thm} \label{th:stable_H}
	Let $0<\alpha<\theta<1$ and assume that $f\in C_b^\alpha$.
	\begin{enumerate}
	 \item[(a)] It holds that $f(X_1) \in (L_2(\PP),\DD)_{\theta,\infty}$, if $\beta \leq 2\alpha/\theta$.
	 \item[(b)] Let $q\in[1,\infty)$ and $\ell\in\{0,1,2,\ldots\}$. 
	            For the function $g^{\alpha,\ell}\in C_b^\alpha$ from Lemma \ref{lem:ciesielski_function} we have that
		\begin{enumerate}
	            \item[(i)] 
	            \[
			   g^{\alpha,\ell}(X_1) \in (L_2(\PP),\DD)_{\theta,q}\quad \text{if and only if}\quad  \beta < 2\alpha/\theta
	            \] and
	           \item[(ii)] 
		    \[
			     g^{\alpha,\ell}(X_1) \in (L_2(\PP),\DD)_{\theta,\infty}\quad \text{if and only if}\quad  \beta \leq 2\alpha/\theta.
	            \]
	        \end{enumerate}
	\end{enumerate}
\end{thm}
\begin{proof}
	(a) {\pink If $\beta\leq 2\alpha$, then $m_{2\alpha/\theta}<\infty$ and the claim follows from Theorem \ref{th:Holder_fractional}(a). Assume now that $\beta>2\alpha$. }  We have
	\begin{align*}
		 \left\| \left\| f(X_1) -f(X_t+\bar{X}_{1-t})   \right\|_{L_2(\PP)} \right\|^2_{L_2(\bar{\PP})}
		& \leq 2\bar{\EE}\EE \left[ |X_{1-t}-\bar{X}_{1-t}|^{2\alpha}\|f\|^2_{C_b^\alpha}  \right] \\
		& \leq 2\bar{\EE} \EE \left[ |(1-t)^{1/\beta}(X_1-\bar{X}_1)|^{2\alpha}\|f\|^2_{C_b^\alpha}  \right] \\
		& \leq 2(1-t)^{2\alpha/\beta}\|f\|^2_{C_b^\alpha} \bar{\EE}\EE \left[ |X_1-\bar{X}_1|^{2\alpha}   \right].
	\end{align*}
	Since the process $X-\bar{X}$ on $\Omega\times\bar{\Omega}$ has the L\'evy measure
	$2\nu$ and $\beta > 2\alpha$, we get that
	\[
		\int_{\{|x|>1\}} |x|^{2\alpha}2\nu(\ud x) = 2\int_{\{|x|>1\}} |x|^{2\alpha-\beta-1} \ud x < \infty,
	\]
	which implies $\bar{\EE}\EE \left[ |X_1-\bar{X}_1|^{2\alpha}   \right]<\infty$ by \cite[Theorem 25.3]{sato}.
	Thus 
	\[
		\left\| \left\| f(X_1) -f(X_t+\bar{X}_{1-t})   \right\|_{L_2(\PP)} \right\|^2_{L_2(\bar{\PP})} \leq C (1-t)^{2\alpha/\beta}
	\]
	for all $t\in(0,1)$ for some $C\in(0,\infty)$ and the claim (a) follows from Lemma \ref{lemma:characterization_for_fractional_smoothness}.\\
	The ''if''-parts of (b) follow from (a) and \eqref{eq:lexicographical_order}.
	By Lemma \ref{lem:ciesielski_function}, the function $g^{\alpha,\ell}$ satisfies \eqref{eq:f_assumption} with $[a,b] = [k2^{-\ell},(k+1)2^{-\ell}]$, 
	$r=2^{-\ell-3}$, $c_2=2^{-\ell}2^{8\alpha-10}$ and $\eta = 2\alpha$. Thus, Lemma \ref{lem:stable_lowerBound} implies that
	\[
		 \left\| \left\| f(X_1) -f(X_t+\bar{X}_{1-t})   \right\|_{L_2(\PP)} \right\|^2_{L_2(\bar{\PP})}
		 \geq c(1-t)^{2\alpha/\beta}
	\]
        for some $c>0$, and with the use of Lemma \ref{lemma:characterization_for_fractional_smoothness} this proves the ''only if''-parts of (b).
\end{proof}

	\begin{thm}\label{th:stable_f} Let $f\in NBV$.
	\begin{enumerate}
		\item[(a)]  If $\beta < 1$, then $f(X_1)\in \DD$.
		\item[(b)]  If $\beta = 1$, then $f(X_1) \in (L_2(\PP),\DD)_{\theta,q}$ for all $\theta\in(0,1)$ and $q\in[1,\infty]$. 
		\item[(c)]  Let $\theta\in(0,1)$.
		 If $\beta \leq 1/\theta$, then $f(X_1) \in (L_2(\PP),\DD)_{\theta,\infty}$.
	\end{enumerate}
\end{thm}
\begin{proof}
	(a) The claim follows from Theorem \ref{th:BV_gives_DD}(a). \\
	(b) The claim follows from Theorem \ref{th:BV_fractional} and \eqref{eq:lexicographical_order}.\\
	(c) If $\beta\leq 1$, then the claim follows from (b). Assume that $\beta>1$. We have
	\begin{align*}
		&  \left\| \left\| f(X_1) -f(X_t+\bar{X}_{1-t})   \right\|_{L_2(\PP)} \right\|^2_{L_2(\bar{\PP})}\\
		& = \bar{\EE}\EE \left[ \left( \int_\RR \one_{[u,\infty)}(X_1) - \one_{[u,\infty)}(X_t+\bar{X}_{1-t}) \mu_f(\ud u) \right)^2\right]\\
		& \leq |\mu_f|(\RR) \int_\RR \bar{\EE}\EE \left[ \one_{[u-X_{1-t},u-\bar{X}_{t-1})\cup[u-\bar{X}_{1-t},u-X_{t-1})}(X_t)\right]\mu_f(\ud u)\\
		& \leq |\mu_f|^2(\RR)\|p_t\|_\infty \bar{\EE}\EE \left[ |\bar{X}_{1-t} - X_{1-t}|  \right]\\
		& = |\mu_f|^2(\RR)t^{-1/\beta}\|p_1\|_\infty \bar{\EE}\EE \left[ (1-t)^{1/\beta}|\bar{X}_1 - X_1|  \right]\\
		& \leq (1-t)^{1/\beta}|\mu_f|^2(\RR)t^{-1/\beta}\|p_1\|_\infty \bar{\EE}\EE \left[ |\bar{X}_1 - X_1|  \right].
	\end{align*}
	Since the process $X-\bar{X}$ has the L\'evy measure
	$2\nu$ and
	\[
		\int_{\{|x|>1\}} |x|2\nu(\ud x) = 2\int_{\{|x|>1\}} |x|^{-\beta} \ud x < \infty,
	\]
	for $\beta>1$, we get $\bar{\EE}\EE\left[ |\bar{X}_1-\hat{X}_1| \right]<\infty$ from \cite[Theorem 25.3]{sato}.
	Thus 
	\[
		\left\| \left\| f(X_1) -f(X_t+\bar{X}_{1-t})   \right\|_{L_2(\PP)} \right\|^2_{L_2(\bar{\PP})} \leq C (1-t)^{1/\beta}
	\]
	for all $t\in(1/2,1)$ for some $C\in(0,\infty)$. {\pink When $t\in(0,1/2]$, then
	\[
	        \left\| \left\| f(X_1) -f(X_t+\bar{X}_{1-t})   \right\|_{L_2(\PP)} \right\|^2_{L_2(\bar{\PP})} \leq \|f\|_{BV}^2 \leq \|f\|_{BV}^2 2^{1/\beta}(1-t)^{1/\beta}
	\]} and the claim follows from Lemma \ref{lemma:characterization_for_fractional_smoothness}.
\end{proof}

\begin{thm}\label{th:stable_K} Let $K\in\RR$. 
	\begin{enumerate}
		\item[(a)]  It holds that $\one_{[K,\infty)}(X_1) \in \DD$ if and only if $\beta < 1$.
		\item[(b)]  It holds that $\one_{[K,\infty)}(X_1) \in (L_2(\PP),\DD)_{\theta,q}$ for all $\theta\in(0,1)$ and $q\in[1,\infty]$ if and only if $\beta\leq 1$. 
		\item[(c)]  Let $\theta\in(0,1)$ and $q\in[1,\infty)$. Then
		\begin{enumerate}
			\item[(i)] $\one_{[K,\infty)}(X_1) \in (L_2(\PP),\DD)_{\theta,q}$ if and only if $\beta < 1/\theta$  and 
			\item[(ii)] $\one_{[K,\infty)}(X_1) \in (L_2(\PP),\DD)_{\theta,\infty}$ if and only if $\beta \leq 1/\theta$.
		\end{enumerate}
		\item[(d)]  Let $\theta\in(0,1)$ and $q\in[1,\infty)$. For the Brownian motion $B$ we have that
		\begin{enumerate}
			\item[(i)] $\one_{[K,\infty)}(B_1) \in (L_2(\PP),\DD)_{\theta,q}$ if and only if $2 < 1/\theta$ and
			\item[(ii)] $\one_{[K,\infty)}(B_1) \in (L_2(\PP),\DD)_{\theta,\infty}$ if and only if $2 \leq 1/\theta$.
		\end{enumerate}
	\end{enumerate}
\end{thm}
\begin{proof}
	(a) The claim follows from Theorem \ref{th:BV_gives_DD}(a) and the proof of Theorem \ref{th:BV_gives_DD}(b), since by \cite[Theorem 24.10(ii)]{sato} 
	    the continuous density of $X_1$ is strictly positive on the whole real line. \\
	(b) The ''if'' follows from Theorem \ref{th:stable_f} and the ''only if'' follows from (c), since $\beta <1/\theta$ for all $\theta\in(0,1)$.\\
	(c) The ''if''-parts of (i) and (ii) follow from Theorem \ref{th:stable_f}(c) and \eqref{eq:lexicographical_order}.\\
	Fix $r>0$ and $t_0\in(0,1)$. By \eqref{align:K_lower}, the function 
	$\one_{[K,\infty)}$ satisties \eqref{eq:f_assumption} with $[a,b] = [K-r,K+r]$, 
	$c_2=1$ and $\eta = 1$. Thus, Lemma \ref{lem:stable_lowerBound} implies that
	\[
		\left\| \left\| f(X_1) - f(X_t+\bar{X}_{1-t}) \right\|_{L_2(\PP)} \right\|_{L_2(\bar{\PP})}^2 \geq c(1-t)^{1/\beta}
	\]
	for some $c>0$.
	The ''only if''-parts of (c) follow now from Lemma \ref{lemma:characterization_for_fractional_smoothness}.\\
	(d) We choose $E$ and $a$ like in the proof of
  	Lemma \ref{lemma:characterization_for_fractional_smoothness} on the corresponding Wiener chaos. The claim follows from Lemma \ref{lemma:thm_geisshujo} 
  	and the proof in \cite[Example 4.7]{geiss-toivola1}, where it is shown that $(Ta)'(t) \sim (1-t)^{-1/2}$.
\end{proof}


%
%

%
%
%


\appendix
\section{Appendix}


The reiteration theorem states that  $(A_0,A_1)_{\eta\theta,q} = (A_0,(A_0,A_1)_{\eta,\infty})_{\theta,q}$ for all $\eta,\theta\in(0,1)$ and $q\in[1,\infty]$
with equivalent norms. In the following lemma we compute explicit constants for the equivalence of the norms for $q=\infty$.

\begin{lem} \label{lemma:constantA}
	Let $(A_0,A_1)$ be a compatible couple and $\eta,\theta\in(0,1)$. Then
	\[ 
	      \|f\|_{(A_0,A_1)_{\eta\theta,\infty}} \leq \|f\|_{(A_0,(A_0,A_1)_{\eta,\infty})_{\theta,\infty}} \leq 3\|f\|_{(A_0,A_1)_{\eta\theta,\infty}}
	\]
	for all $f\in (A_0,A_1)_{\eta\theta,\infty} = \left(A_0,(A_0,A_1)_{\eta,\infty} \right)_{\theta,\infty}$.
\end{lem}
\begin{proof}
	First inequality: Let $t>0$ and $\varepsilon>0$. There exist $f_0, g_0 \in A_0$, $g\in (A_0,A_1)_{\eta,\infty}$ and $g_1 \in A_1$ such that
	$f= f_0 + g = f_0 + g_0 + g_1$ and
	\begin{align*}
		K(f,t^\eta; A_0,(A_0,A_1)_{\eta,\infty})
		  & \geq \|f_0\|_{A_0} + t^\eta \|g\|_{(A_0,A_1)_{\eta,\infty}} - \frac{\varepsilon}{2}                                       \\
		  & \geq \|f_0\|_{A_0} + t^\eta t^{-\eta} \left( \|g_0\|_{A_0} + t\|g_1\|_{A_1} -\frac{\varepsilon}{2} \right) - \frac{\varepsilon}{2} \\
		  & \geq \|f_0 + g_0\|_{A_0}  + t\|g_1\|_{A_1} -\varepsilon                                                                   \\
		  & \geq K(f,t; A_0,A_1) - \varepsilon.                                                                                       
	\end{align*}
	Thus
	\begin{align*}
		\|f\|_{(A_0,(A_0,A_1)_{\eta,\infty})_{\theta,\infty}}
		  & = \sup_{t>0} (t^{\eta})^{-\theta}K(f,t^\eta; A_0,(A_0,A_1)_{\eta,\infty}) \\
		  & \geq \sup_{t>0} t^{-\eta\theta} K(f,t; A_0,A_1)                           \\
		  & = \|f\|_{(A_0,A_1)_{\eta\theta,\infty}}.                                  
	\end{align*}
	Second inequality: Let $f\in (A_0,A_1)_{\eta\theta,\infty}$ and $\varepsilon>0$. For all $t>0$ we find $g_t\in A_0$ and $h_t\in A_1$ such that
	$f=g_t+h_t$ and 
	\[
		\|g_t\|_{A_0} + t\|h_t\|_{A_1} \leq K(f,t;A_0,A_1) + \frac{\varepsilon}{2} t^{\eta\theta}.
	\] 
	Then
	\begin{align*} \label{align:estimates}
		  & K(g_t,s;A_0,A_1) \leq \|g_t\|_{A_0} \leq K(f,t;A_0,A_1)+ \frac{\varepsilon}{2} t^{\eta\theta} \text{ and }                   \\
		  & K(h_t,s; A_0, A_1) \leq s\|h_t\|_{A_1} \leq \frac{s}{t} \left[ K(f,t;A_0,A_1)+ \frac{\varepsilon}{2} t^{\eta\theta} \right]
	\end{align*}
	for all $s\in(0,\infty)$. These inequalities
	give, keeping in mind that $h_t=f-g_t$, that
	\begin{align*}
		t^\eta\|h_t\|_{(A_0,A_1)_{\eta,\infty}}
		  & = t^\eta \sup_{s>0} s^{-\eta} K(h_t,s;A_0,A_1)                                                                                                     \\
		  & \leq \left( \sup_{0<s\leq t} \left( \frac{s}{t}\right)^{-\eta} \frac{s}{t} \left[ K(f,t;A_0,A_1) + \frac{\varepsilon}{2}t^{\eta\theta} \right] \right)      
		\vee \\
		  & \quad \left( \sup_{s\geq t} \left(\frac{s}{t}\right)^{-\eta} \left[ K(f,s;A_0,A_1) + K(g_t,s;A_0,A_1)  \right]  \right)                            \\
		  & \leq \left( K(f,t;A_0,A_1) + \frac{\varepsilon}{2}t^{\eta\theta}  \right) \vee                                                                              \\
		  & \quad \left( \sup_{s\geq t} \left(\frac{s}{t}\right)^{-\eta} \left[ K(f,s;A_0,A_1) + K(f,t;A_0,A_1) + \frac{\varepsilon}{2}t^{\eta\theta}  \right]  \right) \\
		  & \leq K(f,t;A_0,A_1) + \frac{\varepsilon}{2}t^{\eta\theta} + \sup_{s\geq t} \left(\frac{s}{t}\right)^{-\eta\theta} K(f,s;A_0,A_1).                           
	\end{align*}
	We obtain
	\begin{align*}
		  & \|f\|_{(A_0,(A_0,A_1)_{\eta,\infty})_{\theta,\infty}}                                                                                                     \\*
		  & = \sup_{t>0} \left(t^\eta\right)^{-\theta} K\left(f,t^\eta; A_0,(A_0,A_1)_{\eta,\infty}\right)                                                            \\
		  & \leq \sup_{t>0} t^{-\eta\theta} \left( \|g_t\|_{A_0} + t^\eta\|h_t\|_{(A_0,A_1)_{\eta,\infty}}  \right)                                                   \\
		  & \leq \sup_{t>0}  t^{-\eta\theta} \left( 2 K(f,t;A_0,A_1)+ \varepsilon t^{\eta\theta} + \sup_{s\geq t} \left(\frac{s}{t}\right)^{-\eta\theta} K(f,s;A_0,A_1)\right) \\
		  & \leq 3 \sup_{s>0} s^{-\eta\theta}K(f,s;A_0,A_1) + \varepsilon                                                                                                      \\*
		  & = 3\|f\|_{(A_0,A_1)_{\eta\theta,\infty}} + \varepsilon.                                                                                                            
	\end{align*}
\end{proof}



\begin{lem}\label{lemma:holder_interpolationA}
	Let $\alpha\in(0,1)$. Then $C^\alpha_b = (B(\RR), Lip)_{\alpha,\infty}$ with 
	\[ 
		\|\cdot\|_{C^\alpha_b} \leq 3 \|\cdot\|_{(B(\RR), Lip)_{\alpha,\infty}}
					\leq   6 \|\cdot\|_{C^\alpha_b}.
	\] 
\end{lem}
\begin{proof} 
	First inequality: Let $f\in (B(\RR), Lip)_{\alpha,\infty}$ and $\varepsilon>0$. For all $t>0$ we find $f_t\in Lip$ such that
	\[
		t^{-\alpha} \left( \|f-f_t\|_\infty + t \|f_t\|_{Lip}  \right) \leq \|f\|_{(B(\RR), Lip)_{\alpha,\infty}} + \varepsilon.
	\]
	Let $x\neq y \in\RR$ and $t=|x-y|>0$. By the triangle inequality we have
	\begin{align*}
		\frac{|f(x)- f(y)|}{|x-y|^\alpha}
		& \leq |x-y|^{-\alpha} \big( |f(x)-f_t(x)| + |f(y)-f_t(y)| + |f_t(x) - f_t(y) |  \big) \\
		& \leq t^{-\alpha} \left( 2\|f-f_t\|_{\infty} + t\|f_t\|_{Lip}  \right)                \\
		& \leq 2 \left( \|f\|_{(B(\RR), Lip)_{\alpha,\infty}} + \varepsilon  \right).                 
	\end{align*}
	It also holds that
	\[ 
		\|f\|_{\infty} \leq \|f-f_1\|_\infty + \|f_1\|_\infty \leq \|f\|_{(B(\RR), Lip)_{\alpha,\infty}} + \varepsilon,
	\]
	so that
	\[
		\|f\|_{C_b^{\alpha}} = \|f\|_\infty + \sup_{x\neq y} \frac{ |f(x) - f(y)| }{ |x-y|^{\alpha}} \leq 3 \|f\|_{(B(\RR), Lip)_{\alpha,\infty}}.
	\]					
	Second inequality: 
	Let $f\in C^\alpha_b$ and $t>0$ and define $f_t$ so that $f_t(kt) = f(kt)$ for $k\in\mathbbm{Z}$ and $f_t$ is linear on each interval $[kt,(k+1)t]$, $k\in\mathbbm{Z}$.
	Then for $x\in[kt,(k+1)t]$ there is $s\in[0,1]$ such that $f_t(x) = sf(kt) + (1-s)f((k+1)t)$ and we get that
	\begin{align*}
		    \|f-f_t\|_\infty
		& = \sup_{k\in\mathbbm{Z}} \sup_{x\in[kt,(k+1)t]} |f(x)-f_t(x)| \\
		& \leq \sup_{k\in\mathbbm{Z}} \sup_{x\in[kt,(k+1)t]} \big( s|f(x)-f(kt)| + (1-s)|f(x)-f((k+1)t)|  \big)\\
		& \leq  \sup_{|x-y|\leq t} |f(x)-f(y)|\\
		& \leq t^\alpha \|f\|_{C_b^\alpha}.
	\end{align*}
	For the function $f_t$ it holds for $0<t\leq 1$ that 
	\begin{align*}
	\|f_t\|_{Lip}
	    &= \|f_t\|_\infty + \sup_{x\neq y} \frac{|f_t(x)-f_t(y)|}{|x-y|}\\
	    & \leq \|f\|_\infty + \sup_{k\in\mathbbm{Z}} \frac{|f(kt)-f((k+1)t)|}{t} \\
	    & \leq \|f\|_\infty + t^{\alpha-1} \sup_{x\neq y} \frac{|f(x)-f(y)|}{|x-y|^\alpha} \\
	    & \leq t^{\alpha-1} \|f\|_{C_b^\alpha}.
	\end{align*}
	Hence we obtain that
	  \begin{align*}
	   \|f\|_{(B(\RR), Lip)_{\alpha,\infty}}
	   & \leq \left[ \sup_{0<t\leq1} t^{-\alpha} \left( \|f-f_t\|_\infty + t \|f_t\|_{Lip}  \right) \right] \vee \sup_{t\geq 1}t^{-\alpha}\|f\|_\infty \\
	   & \leq  2\|f\|_{C_b^\alpha}.
	  \end{align*}

\end{proof}



\end{document}